\documentclass{compositio}

\usepackage{todonotes}

\usepackage{tipa}
\usepackage{dsfont}
\usepackage{tikz-cd}
 
\usepackage{xr}
 
\usepackage{amsxtra}
\usepackage{amsmath}
\usepackage{amssymb}
\usepackage{amsfonts}
\usepackage[all,2cell]{xy}
\usepackage{mathrsfs}
\usepackage{amsthm}
\usepackage{enumitem}
\usepackage{hyperref}
\usepackage{subfigure}

\usepackage{euscript}

\newtheorem{thm}[equation]{Theorem}
\newtheorem{cor}[equation]{Corollary}
\newtheorem{lem}[equation]{Lemma}
\newtheorem{prop}[equation]{Proposition}

\newtheoremstyle{example}{\topsep}{\topsep}%
     {}
     {}
     {\bfseries}
     {.}
     {2pt}
     {\thmname{#1}\thmnumber{ #2}\thmnote{ #3}}

   \theoremstyle{example}
\newtheorem{defi}[equation]{Definition}
\newtheorem{rem}[equation]{Remark}

\newtheorem{exa}[equation]{Example}
\newtheorem{ex}[equation]{Example}

\def\on{\operatorname}
\def\op{{\operatorname{op}}}

\def\bDelta{ {\bf \Delta}}
\def\bLambda{ {\bf \Lambda}}
\def\pLambda{ {\Lambda_{\infty}}}
\def\apLambda{ {\Lambda_{\infty}^{+}}}

\def\Fun{\operatorname{Fun}}

\def\h{\operatorname{h}\!}

\def\colim{\operatorname*{colim}}

\def\Hom{\operatorname{Hom}}
\def\GL{\operatorname{GL}}
\def\RR{\mathbb{R}}

\def\S{\EuScript S}

\def\I{{\mathcal I}}

\def\HP{ {\operatorname{HP}}}

\def\J{{\mathcal J}}

\def\h{\operatorname{h}\!}

\def\Map{\operatorname{Map}}

\def\N{\operatorname{N}}

\def\Sp{{\on{Sp}}}

\def\Cat{{\mathcal Cat}}

\def\Rib{ {\mathcal Rib}}

\def\cm{\langle m \rangle}
\def\cn{\langle n \rangle}

\def\Set{ {\mathcal Set}}

\def\A{ {\EuScript A}}
\def\AA{ {\mathbb A}}
\def\P{{\EuScript P}}

\def\C{{\EuScript C}}
\def\D{{\EuScript D}}

\def\J{ {\EuScript J}}

\def\Set{ {\mathcal Set}}

\def\<<{\langle {}\hskip -.1cm {}\langle}
\def\>>{\rangle \hskip -.1cm \rangle}

\def\Ft{\on{\bar{F}}}
\def\Ff{\on{F}}

\setlist[enumerate,1]{label=(\arabic{*})}
\setlist[enumerate,2]{label=(\alph{*})}
\setlist[enumerate,3]{label=(\roman{*})}

\def\centerarc[#1](#2)(#3:#4:#5)  { \draw[#1] ($(#2)+({#5*cos(#3)},{#5*sin(#3)})$) arc (#3:#4:#5); }
\tikzstyle{dot}=[draw,circle,fill=black,inner sep=0,minimum size=4pt]
\tikzset{
    partial ellipse/.style args={#1:#2:#3}{
        insert path={+ (#1:#3) arc (#1:#2:#3)}
    }
}

\def\lra{\longrightarrow}
\def\Zt{\mathbb{Z}/2\mathbb{Z}}
\def\ZZ{\mathbb{Z}}

\def\Cat{{\EuScript Cat}}

\def\k{\mathbf k}
\def\dgcat{{\EuScript Cat}_{\on{dg}}}
\def\Grp{\on{Grp}}
\def\dgcatt{{\EuScript Cat}_{\on{dg}}^{(2)}}
\def\Lqe{\on{L}_{\on{qe}}(\dgcat)}
\def\Lmo{\on{L}_{\on{mo}}(\dgcat)}

\def\Lmot{\on{L}_{\on{mo}}(\dgcatt)}

\def\UHom{\underline{\on{Hom}}}

\def\Mod{\on{Mod}}

\def\Perf{\on{Perf}}
\def\coh{\on{coh}}

\def\MF{\on{MF}}

\def\ev{\on{ev}}

\def\id{\on{id}}

\def\cm{\langle m \rangle}
\def\cn{\langle n \rangle}

\title{$\AA^1$-homotopy invariants of topological Fukaya categories of surfaces}
\author{Tobias Dyckerhoff }
\email{dyckerho@math.uni-bonn.de}
\address{Hausdorff Center for Mathematics\\Endenicher Allee 62\\53115 Bonn\\Germany}

\begin{document}

\begin{abstract}
	We provide an explicit formula for localizing $\AA^1$-homotopy invariants of topological
	Fukaya categories of marked surfaces. Following a proposal of Kontsevich, this
	differential $\ZZ$-graded category is defined as global sections of a constructible cosheaf
	of dg categories on any spine of the surface. Our theorem utilizes this sheaf-theoretic
	description to reduce the calculation of invariants to the local case when the surface is a
	boundary-marked disk. At the heart of the proof lies a theory of localization for
	topological Fukaya categories which is a combinatorial analog of Thomason-Trobaugh's theory
	of localization in the context of algebraic $K$-theory for schemes.
\end{abstract}

\maketitle

\tableofcontents

\addcontentsline{toc}{section}{Introduction}

\vfill\eject

\numberwithin{equation}{section}

\section*{Introduction}

According to Kontsevich's proposal \cite{kontsevich:symplectic}, Fukaya categories of Stein manifolds 
can be described as global sections of a constructible cosheaf of dg categories on a possibly singular 
Lagrangian spine onto which the manifold retracts. Various approaches to this proposal have been developed,
see, e.g. \cite{nadler:morse,bocklandt:noncommutative,sibilla:ribbon,haiden:stability,
nadler:cyclic,dk:triangulated,nadler:arboreal}.


The specific case we focus on in this work is the following: Let $S$ be a compact connected Riemann
surface, possibly with boundary, and let $M \subset S$ be a finite nonempty subset of marked points, so that the
complement $S \setminus M$ is a Stein manifold. Any spanning graph $\Gamma$ in $S \setminus M$
provides a Lagrangian spine.  In \cite{dk:triangulated}, the language of {\em cyclic $2$-Segal
spaces} was used to realize Kontsevich's proposal in this situation. For every commutative ring $k$, 
this theory produces a constructible cosheaf of $k$-linear differential $\Zt$-graded categories on
any spanning graph $\Gamma$, shows that the dg category $F(S,M;k)$ of global
sections is independent of the chosen graph, and implies a coherent action of the mapping class
group of the surface on $F(S,M;k)$. It is expected that the resulting dg category is Morita equivalent
to a variant of the wrapped Fukaya category of the surface. We refer to $F(S,M;k)$ as the $k$-linear
{\em topological Fukaya category of the surface $(S,M)$}. If the surface $S \setminus M$ is equipped
with a framing then a {\em paracyclic} version of the above construction can be used to define a
differential $\ZZ$-graded lift of $F(S,M;k)$ (cf.  \cite{lurie:rotation, dk:crossed}). The first part
of this paper implements these constructions in the framework of $\infty$-categories which provides
the flexibility necessary for our purposes.

A functor $H$ defined on the category of small $k$-linear dg categories with values in a stable
$\infty$-category $\C$ is called 
\begin{enumerate}
	\item {\em localizing} if $H$ inverts Morita equivalences and sends exact sequences of dg
		categories to exact sequences in $\C$.
	\item {\em $\AA^1$-homotopy invariant} if, for every dg category $A$, the functor $H$ maps $A \to A[t]$ to
		an equivalence in $\C$.
\end{enumerate}
Examples of localizing $\AA^1$-homotopy invariants are provided by periodic cyclic homology over
a field of characteristic $0$, homotopy $K$-theory, $K$-theory with finite coefficients, and topological $K$-theory
over ${\mathbb C}$ (cf. \cite{keller:invariance,tabuada-vdbergh,blanc}). 
The main result of this work is the following:

\begin{thm}\label{thm:main_intro} Let $H$ be a localizing $\AA^1$-homotopy invariant with values in a stable
	$\infty$-category $\C$, and let $(S,M)$ be a stable marked surface where $S \setminus M$ is equipped with a framing.
	Define $E = H(k)$ to be the object of $\C$ obtained by applying $H$ to the dg category with
	one object and endomorphism ring $k$. Then there is an equivalence 
	\[
		H(F(S,M;k)) \simeq E(S,M)[-1]
	\]
	where the right-hand side denotes the relative homology of the
	pair $(S,M)$ with coefficients in $E[-1]$.
\end{thm}

As a concrete example, we have:

\begin{cor} Let $k$ be a field of characteristic $0$. We have the formulas
	\begin{align*}
		\HP_0( F(S,M;k)) & \cong  H_1(S,M;k)\\
		\HP_1( F(S,M;k)) & \cong  H_2(S,M;k)
	\end{align*}
	for periodic cyclic homology over $k$.
\end{cor}

Our proof strategy is as follows:
We prove a Mayer-Vietoris type statement for localizing $\AA^1$-homotopy invariants of
topological Fukaya categories. Roughly speaking, using the local nature of the topological
Fukaya category, the result allows us to reduce the calculation of such an invariant
to the case when the surface is a boundary-marked disk. The arguments we use are inspired by methods
in Thomason-Trobaugh's algebraic $K$-theory of schemes and may be of independent interest: We
establish localization sequences for topological Fukaya categories which are analogous to the ones
for derived categories of schemes appearing in \cite{thomason-trobaugh}. Similar localization
techniques for Fukaya categories play a role in \cite{haiden:stability}, where the Grothendieck
group is computed, and the forthcoming work \cite{sibilla}.
The Mayer-Vietoris theorem expresses $H(F(S,M))$ as a state sum of a coparacyclic $1$-Segal object
in a stable $\infty$-category. We conclude by proving a delooping statement for such objects which
allows us to explicitly compute the state sum. Finally, I would like to point out that Jacob Lurie
has communicated to me a version of Theorem \ref{thm:main_intro} which does not assume $\AA^1$-homotopy
invariance but assumes that the surface has no internal marked points.

We will use the language of $\infty$-categories and refer to \cite{lurie:htt} as a general
reference.\\

\noindent
{\bf Acknowledgements.} I would like to thank Chris Brav, Mikhail Kapranov, Jacob Lurie, Pranav Pandit, Paul
Seidel, and Nicolo Sibilla for inspiring conversations on the subject of this paper.

\section{Cyclic and paracyclic $2$-Segal objects}

We begin with a translation of some constructions in \cite{dk:triangulated} into the context of
$\infty$-categories. 

\subsection{The Segal conditions}
\label{sec:segal}

We start by formulating the Segal conditions (cf. \cite{segal.categories, rezk, dk:highersegal}). 

\begin{rem}\label{rem:bDelta} Let $\bDelta$ denote the category of finite nonempty linearly ordered
	sets. This category contains the simplex category $\Delta$ as the full subcategory spanned
	by the collection of standard ordinals $\{ [n], n \ge 0\}$. For every object of $\bDelta$, there exists a {\em
	unique} isomorphism with an object of $\Delta$ so that we can identify any diagram in
	$\bDelta$ with a unique diagram in $\Delta$. With this in mind, we may use arbitrary
	finite nonempty linearly ordered sets to describe diagrams in $\Delta$ without ambiguity.  
\end{rem}

\begin{defi}
	Let $\C$ be an $\infty$-category, and let $X^{\bullet}: \N(\Delta) \to \C$ be a cosimplicial object in $\C$.
\begin{enumerate} 
	\item The cosimplicial object $X^{\bullet}$ is called {\em $1$-Segal} if, for every $0 < k < n$, the resulting diagram
		\[
			\begin{tikzcd}
				X^{\{k\}} \arrow{r}\arrow{d} & X^{\{k,k+1,\dots,n\}} \arrow{d}\\
				X^{\{0,1,\dots,k\}} \arrow{r} & X^{\{0,1,\dots,n\}}
			\end{tikzcd}
		\]
		in $\C$ is a pushout diagram.

	\item  Let $P$ be a planar convex polygon with vertices labelled cyclically by the set $\{0,1,\dots,n\}$, $n \ge 3$. Consider
		a diagonal of $P$ with vertices labelled by $i < j$ so that we obtain a subdivision of $P$ into two
		subpolygons with vertex sets $\{0,1,\dots,i,j,\dots,n\}$ and $\{i,i+1,\dots,j\}$,
		respectively. The resulting triple of numbers $0 \le i < j \le n$ is called a
		{\em polygonal subdivision}.
		\begin{enumerate}[label=(\roman *)]
			\item \label{en:2co} The cosimplicial object $X^{\bullet}$ is called {\em $2$-Segal} if, for 
				every polygonal subdivision $0 \le i < j \le n$, the resulting diagram
					\begin{equation}\label{eq:2segal}
						\begin{tikzcd}
							X^{\{i,j\}} \arrow{r}\arrow{d} & X^{\{i,i+1,\dots,j\}} \arrow{d}\\
							X^{\{0,1,\dots,i,j,\dots,n\}} \arrow{r} & X^{\{0,1,\dots,n\}}
						\end{tikzcd}
					\end{equation}
				in $\C$ is a pushout diagram. 
			\item We say $X^{\bullet}$ is a {\em unital} $2$-Segal object if, in addition to
				\ref{en:2co}, for every $0 \le k < n$, the diagram 
				\begin{equation}\label{eq:unital}
						\begin{tikzcd}
							X^{\{k,k+1\}} \arrow{r}\arrow{d} & X^{\{0,1,\dots,n\}} \arrow{d}\\
							X^{\{k\}} \arrow{r} & X^{\{0,1,\dots,k,k+2,\dots,n\}}
						\end{tikzcd}
					\end{equation}
				in $\C$ is a pushout diagram.
		\end{enumerate}
\end{enumerate}
\end{defi}

\begin{prop}\label{prop:12segal} Every $1$-Segal cosimplicial object $X^{\bullet}: \N(\Delta) \to \C$ is unital
	$2$-Segal.
\end{prop}
\begin{proof} Given a polygonal subdivision $0 \le i \le j \le n$, we augment the corresponding
	square \eqref{eq:2segal} to the diagram 
	\[
						\begin{tikzcd}
							X^{\{i\}} \amalg X^{\{j\}} \arrow{r}\arrow{d} & X^{\{i,j\}} \arrow{r}\arrow{d} & X^{\{i,i+1,\dots,j\}} \arrow{d}\\
							X^{\{0,1,\dots,i\}} \amalg
							X^{\{j,j+1,\dots,n\}} \arrow{r} &
							X^{\{0,1,\dots,i,j,\dots,n\}} \arrow{r} &
							X^{\{0,1,\dots,n\}}.
						\end{tikzcd}
	\]
	The $1$-Segal condition on $X^{\bullet}$ implies that the left-hand square and outer
	square are pushouts. Hence, by \cite[4.4.2.1]{lurie:htt}, the right-hand square is a
	pushout. Similarly, to obtain unitality, we augment the square \eqref{eq:unital} to the
	diagram
	\[
						\begin{tikzcd}
							X^{\{k\}} \amalg X^{\{k+1\}} \arrow{d}
							\arrow{r} & X^{\{0,1,\dots,k\}} \amalg
							X^{\{k+1,k+2,\dots,n\}} \arrow{d}\\
							X^{\{k,k+1\}} \arrow{r}\arrow{d} & X^{\{0,1,\dots,n\}} \arrow{d}\\
							X^{\{k\}} \arrow{r} &
							X^{\{0,1,\dots,k,k+2,\dots,n\}}.
						\end{tikzcd}
	\]
	The $1$-Segal condition on $X^{\bullet}$ implies that the top and outer squares are pushouts
	so that by loc. cit. the bottom square is a pushout.
\end{proof}

%

\subsection{Cyclically ordered sets and ribbon graphs}

Let $S$ be a compact oriented surface, possibly with boundary $\partial S$, together with a chosen finite 
subset $M \subset S$ of marked points. We call $(S,M)$ stable if 
\begin{enumerate}
	\item every connected component of $S$ has at least one marked point,
	\item every connected component of $\partial S$ has at least one marked point,
	\item every connected component of $S$ which is homeomorphic to the $2$-sphere has at least two marked
		points.
\end{enumerate}
In this section, we develop a categorified state sum formalism based on a version of the well-known
combinatorial description of stable oriented marked surfaces $(S,M)$ in terms of Ribbon graphs.

\subsubsection{Cyclically ordered sets}
\label{sec:cyclically}

In Section \ref{sec:segal}, we enlarged the simplex category $\Delta$ to the equivalent
category $\bDelta$ of finite nonempty linearly ordered sets. In analogy, we introduce the category of nonempty finite
cyclically ordered sets $\bLambda$ following \cite{dk:crossed} which plays the same role for the
cyclic category $\Lambda$. 

Let $J$ be a finite nonempty set. We define a {\em cyclic order} on $J$ to be a transitive action
of the group $\ZZ$. Note that any such an action induces a simply transitive action of the
group $\ZZ/N\ZZ$ on $J$ where $N$ denotes the cardinality of the set $J$.

\begin{ex} Let $I$ be a finite nonempty linearly ordered set. We obtain a cyclic order on $I$ as
	follows: Let $i_0 < i_1 < \dots < i_n$ denote the elements of $I$. We set, for $0 \le k <n$, $i_k + 1 =
	i_{k+1}$, and $i_n +1 = i_0$. We call the resulting cyclic order on $I$ the
	{\em cyclic closure} of the given linear order. We denote the cyclic closure of the standard ordinal
	$[n]$ by $\cn$.
\end{ex}

\begin{ex}\label{exa:lex} 
	More generally, let $f: J \to J'$ be a map of finite nonempty sets. Assume that $J'$
	carries a cyclic order and that every fiber of $f$ is equipped with a linear order. We
	define the {\em lexicographic cyclic order} on $J$ as follows: Let $j \in J$. If $j$ is not
	maximal in its fiber then we define $j+1$ to be the successor to $j$ in its fiber. If $j$ is
	maximal in its fiber, then we define $j+1$ to be the minimal element of the successor fiber
	(The cyclic order on $J'$ induces a cyclic order on the fibers of $f$ where we simply skip
	empty fibers). 
\end{ex}

A {\em morphism} $J \to J'$ of cyclically ordered sets consists of
\begin{enumerate}
	\item a map $f: J \to J'$ of underlying sets,
	\item the choice of a linear order on every fiber of $J'$,
\end{enumerate}
such that the cyclic order on $J$ is the lexicographic order from Example \ref{exa:lex}.
We denote the resulting category of cyclically ordered sets by $\bLambda$.
Given a cyclically ordered set $J$, we define the {\em set of interstices} 
\[
	J^{\vee} = \Hom_{\bLambda}(J,\langle 0 \rangle)
\]
which, by definition, is the set of linear
orders on $J$ whose cyclic closure agrees with the given cyclic order on $J$. If the set $J$ has cardinality
$n+1$, then we may identify $J^{\vee}$ with the set of isomorphisms $J \to \cn$ in $\bLambda$. The
$\ZZ$-action on $\cn$ induces a $\ZZ$-action on $J^{\vee}$ which defines a cyclic order. 

\begin{prop} The association $J \mapsto J^{\vee}$ extends to an equivalence of categories
	$\bLambda^{\op} \to \bLambda$. 
\end{prop}
\begin{proof}
	Given a morphism $f: J \to J'$ in $\bLambda$, we have to define a dual $f^{\vee}:
	(J')^{\vee} \to J^{\vee}$. The datum of $f$ includes a choice of linear order on each fiber
	of $f$. Given a linear order on $J'$, we can form the lexicographic linear order on $J$ by
	a linear analog of the construction in Example \ref{exa:lex}. This defines the map
	$f^{\vee}$ on underlying sets. We further have to define a linear order on the fibers of
	$f^{\vee}$. Given linear orders $h: J' \cong [n]$ and $h':J' \cong [n]$ such that
	$f^{\vee}(h) = f^{\vee}(h')$, we fix any $j \in J$ and declare $h \le h'$ if $h(j) \le
	h'(j)$. This defines a linear order on each fiber of $f^{\vee}$ which does in fact not
	depend on the chosen element $j \in J$. 
	To verify that $J \mapsto J^{\vee}$ is an equivalence, we observe that the double dual is
	naturally equivalent to the identity functor: An element $j \in J$ determines a linear order
	on $J^{\vee}$ by declaring, for $h: J \cong [n]$ and $h': J \cong [n]$, $h \le h'$ if $h(j)
	\le h'(j)$. We leave to the reader the verification that this association defines an
	isomorphism
	\[
		J \to (J^{\vee})^{\vee}
	\]
	in $\bLambda$ which extends to a natural isomorphism between the identity functor and the double dual.
\end{proof}

We refer to the equivalence $\bLambda^{\op} \to \bLambda$ as {\em interstice duality}. The following
Lemma will be important for the interplay between cyclic $2$-Segal objects and ribbon
graphs.

\begin{lem} \label{lem:pullback} Let $I$,$J$ be finite sets with elements $i \in I$, $j \in J$, and
	consider the pullback square
		\begin{equation}\label{eq:cyclicpullback}
			\begin{tikzcd}
				K \arrow{r} \arrow{d} & J \arrow{d}{q} \\
				I \arrow{r}{p} & \{i,j\} 
			\end{tikzcd}
		\end{equation}
	where the maps $p$ and $q$ are determined by $p^{-1}(i) = \{i\}$ and $q^{-1}(j) = \{j\}$.
	Assume that $K$ is nonempty and that the sets $I$ and $J$ are equipped with cyclic orders. Then the following hold:
	\begin{enumerate}
		\item The above diagram lifts uniquely to a diagram in $\bLambda$ such that the
			induced cyclic orders in $I$ and $J$ are the given ones.
		\item The resulting square in $\bLambda$ is a pullback square.	
	\end{enumerate}
\end{lem}
\begin{proof} Follows by direct inspection.
\end{proof}

\begin{ex}\label{ex:cyclic-pushout} Consider the diagram of linearly ordered sets
\begin{equation}\label{eq:pushout}
	\begin{tikzcd}
		\{i,j\} \arrow{r}\arrow{d}& \arrow{d}\{i,i+1,\dots,j\} \\
		\{0,1,\dots,i,j,\dots,n\}  \arrow{r} & \{0,1,\dots,n\}
	\end{tikzcd}
\end{equation}
corresponding to a polygonal subdivision $0 \le i < j \le n$ of a planar convex polygon as in
Section 
\ref{sec:segal}. Passing to cyclic closures we obtain a diagram in $\bLambda$. By applying
interstice duality we obtain a diagram in $\bLambda$ of the form \eqref{eq:cyclicpullback} which is hence a
pullback diagram. We deduce that the original diagram \eqref{eq:pushout} is a pushout diagram in $\bLambda$. 
Further, the same argument implies that, for every $0 \le k < n$, the diagram
\begin{equation}\label{eq:pushout2}
	\begin{tikzcd}
		\{k,k+1\} \arrow{r}\arrow{d}& \arrow{d}\{0,1,\dots,n\} \\
		\{k\}  \arrow{r} & \{0,1,\dots,k,k+2,\dots,n\}
	\end{tikzcd}
\end{equation}
is a pushout diagram in $\bLambda$.
\end{ex}

\begin{defi} Let $\C$ be an $\infty$-category. A cocyclic object $X: \N(\bLambda) \to \C$ is called
	(unital) $2$-Segal (resp. $1$-Segal) if the underlying cosimplicial object is (unital)
	$2$-Segal (resp. $1$-Segal).
\end{defi}

\begin{rem} \label{rem:2segalpush}
	Note, that the diagram 
	\begin{equation}\label{eq:1segal}
			\begin{tikzcd}
				\{k\} \arrow{r}\arrow{d} & \{k,k+1,\dots,n\} \arrow{d}\\
				\{0,1,\dots,k\} \arrow{r} & \{0,1,\dots,n\}
			\end{tikzcd}
	\end{equation}
	is a pushout diagram in $\bDelta$ and the $1$-Segal condition requires a cosimplicial object
	$X: \Delta \to \C$ to preserve this pushout. In light of this observation, the $2$-Segal condition
	becomes very natural for cocyclic objects: while the cyclic closure of \eqref{eq:1segal} is
	{\em not} a pushout square in $\bLambda$, the squares \eqref{eq:pushout} and
	\eqref{eq:pushout2} are pushout squares. The $2$-Segal condition requires that these pushouts are
	preserved. 
\end{rem}

\subsubsection{Ribbon graphs}

A {\em graph} $\Gamma$ is a pair of finite sets $(H,V)$ equipped with an involution $\tau: H \to H$
and a map $s: H \to V$. The elements of the set $H$ are called {\em halfedges}. We call a halfedge
{\em external} if it is fixed by $\tau$, and {\em internal} otherwise. A pair of $\tau$-conjugate
internal halfedges is called an {\em edge} and we denote the set of edges by $E$. The elements of
$V$ are called {\em vertices}. Given a vertex $v$, the halfedges in the set $H(v) = s^{-1}(v)$ are
said to be {\em incident to $v$}. 

A {\em ribbon graph} is a graph $\Gamma$ where, for every vertex $v$ of $\Gamma$, the set $H(v)$ of
halfedges incident to $v$ is equipped with a cyclic order. We give an interpretation of this datum
in terms of the category of cyclically ordered sets:
Let $\Gamma$ be a graph. We define the incidence category $I(\Gamma)$ to have set of objects given
by $V \cup E$ and, for every internal halfedge $h$, a unique morphism from the vertex $s(h)$ to the
edge $\{h, \tau(h)\}$. We define a functor
\[
	\gamma: I(\Gamma) \lra \Set
\] 
which, on objects, associates to a vertex $v$ the set $H(v)$ and to an edge $e$ the set of halfedges underlying
$e$. To a morphism $v \to e$, given by a halfedge $h \in H(v)$ with $h \in e$, we associate the map
$\pi: H(v) \to e$ which is determined by $\pi^{-1}(h) = \{h\}$ so that $\pi$ maps $H(v) \setminus h$ to $\tau(h)$. 
We call the functor $\gamma$ the {\em incidence diagram} of the graph $\Gamma$.

\begin{prop}\label{prop:ribbon} Let $\Gamma$ be a graph. A ribbon structure on $\Gamma$ is equivalent to a lift
	\[
	\begin{tikzcd}
			 & \bLambda \arrow{d}\\
			I(\Gamma) \arrow{r} \arrow[dashed]{ur} & \Set
	\end{tikzcd}
	\]
	of the incidence diagram of $\Gamma$ where $\bLambda$ denotes the category of cyclically
	ordered sets.
\end{prop}

The advantage of the interpretation of a ribbon structure given in Proposition \ref{prop:ribbon} is
that it facilitates the passage to interstices: Given a ribbon graph $\Gamma$ with corresponding
incidence diagram
\[
	\gamma: I(\Gamma) \to \bLambda
\]
we introduce the {\em coincidence diagram}
\[
  \delta: I({\Gamma})^{\op} \lra \bLambda
\]
obtained by postcomposing $\gamma^{\op}$ with the interstice duality functor $\bLambda^{\op} \to
\bLambda$.

A morphism $(f,\eta): \Gamma \to \Gamma'$ of ribbon graphs consists of 
\begin{enumerate}
	\item a functor $f: I(\Gamma) \to I(\Gamma')$ of incidence categories,
	\item a natural transformation $\eta: f^*\gamma' \to \gamma$ of incidence diagrams.
\end{enumerate}

\begin{ex} \label{ex:contraction} Let $\Gamma$ be a graph and let $e$ be an edge incident to two distinct vertices
	$v$ and $w$. We define a new graph $\Gamma'$ obtained from $\Gamma$ by contracting $e$ as
	follows: The set of halfedges $H'$ is given by $H \setminus e$ and the set of vertices $V'$
	is obtained from $V$ by identifying $v$ and $w$. The involution $\tau$ on $H$ restricts to
	define an involution $\tau'$ on $H'$. We define $s: H' \to V'$ as the composite
	of the restriction of $s: H \to V$ to $H'$ and the quotient map $V \to V'$. 
	
	We obtain a natural functor $f:I(\Gamma) \to I(\Gamma')$ of incidence categories which collapses the objects $v$, $w$ 
	and $e$ to $\overline{v}$. Denoting by $\gamma$ and $\gamma'$ the set-valued incidence
	diagrams, we construct a natural transformation $\eta: f^*\gamma' \to \gamma$ as follows. On
	objects of $I(\Gamma)$ different from $v$, $w$, and $e$, we define $\eta$ to be the identity
	map. To obtain the values of $\eta$ at $v$, $w$, and $e$, note that we have a natural commutative diagram 
	\[
		\begin{tikzcd}
			H(v) \arrow{d} & \\
			e &  \arrow{l}[swap]{\eta_e} \arrow{lu}[swap]{\eta_v} \arrow{ld}{\eta_w} H(\overline{v})\\
			H(w) \arrow{u} & 
		\end{tikzcd}
	\]
	as in Lemma \ref{lem:pullback}. The maps in the diagram determine the values of $\eta$
	as indicated.  

	Finally, assume that $\Gamma$ carries a ribbon structure. Then Lemma \ref{lem:pullback}(1)
	implies that $\Gamma'$ carries a unique ribbon structure such that the natural
	transformation $\eta$ lifts to $\bLambda$-valued incidence diagrams. 
\end{ex}

A morphism of ribbon graphs $\Gamma \to \Gamma'$ as constructed in Example \ref{ex:contraction} is
called an {\em edge contraction}. The following Proposition indicates the relevance of Lemma
\ref{lem:pullback}(2).

\begin{prop}\label{prop:ribkan} Let $\Gamma$ be a ribbon graph and let $(f,\eta): \Gamma \to \Gamma'$ be an edge
	contraction. Then the natural transformation
	\[
		\eta: f^* \gamma' \lra \gamma
	\]
	exhibits $\gamma'$ as a right Kan extension of $\gamma$ along $f$.
\end{prop}
\begin{proof} By the pointwise formula for right Kan extensions, it suffices to verify that, for
	every object $y \in I(\Gamma')$, the natural transformation $\eta$ exhibits $\gamma'(y)$ as the
	limit of the diagram
	\[
		y/f \lra \Lambda, \; (x, y \to f(x)) \mapsto \gamma(x).
	\]
	Unravelling the definitions, this is a trivial condition unless $y$ is the object
	corresponding to the vertex $\overline{v}$ under the contracted edge. For $y = \overline{v}$
	the condition reduces to Lemma \ref{lem:pullback}(2).
\end{proof}

\subsubsection{State sums on ribbon graphs}
\label{sec:statesum}

We introduce a category $\Rib^*$ with objects given by pairs $(\Gamma, x )$ where $\Gamma$ is a
ribbon graph and $x$ is an object of the incidence category $I(\Gamma)$. A morphism 
$(\Gamma, x) \to (\Gamma',y)$ consists of a morphism $(f,\eta): \Gamma \to \Gamma'$ of
ribbon graphs together with a morphism $y \to f(x)$ in $I(\Gamma')$. The category $\Rib^*$ comes
equipped with a forgetful functor $\pi: \Rib^* \to \Rib$ and an evaluation functor
\[
  \ev: \Rib^* \lra \bLambda, \; (\Gamma, x) \mapsto \delta(x)
\]
where $\delta$ denotes the coincidence diagram of $\Gamma$.
Let $\C$ be an $\infty$-category with colimits, and let $X: \N(\bLambda) \to \C$ be a cocyclic object in
$\C$. The functor
\[
  	\rho_X = \N(\pi)_!(X \circ \N(\ev)): \N(\Rib) \lra \C 
\]
is called the {\em state sum functor of $X$}. Here, $\N(\pi)_!$ denotes the $\infty$-categorical
left Kan extension defined in \cite[4.3.3.2]{lurie:htt}. For a ribbon graph $\Gamma$, the object
$X({\Gamma}) := \rho_X(\Gamma)$ is called the {\em state sum of $X$ on $\Gamma$}.

\begin{prop}\label{prop:pointwise}
	The state sum of $X$ on $\Gamma$ admits the formula
	\[
		X(\Gamma) \simeq \colim X \circ \N(\delta)
	\]
	where $\delta: I(\Gamma)^{\op} \to \bLambda$ denotes the coincidence diagram of $\Gamma$.
\end{prop}
\begin{proof} By the pointwise formula for left Kan extensions, we have
    \[
	    X(\Gamma) = \colim X \circ \N(\ev)_{|\N(\pi/\Gamma)}.
    \]
    The nerve of the functor 
    \[
	    I(\Gamma)^{\op} \lra \pi/\Gamma, x \mapsto ((\Gamma,x), \Gamma \overset{\id}{\to}
	    \Gamma)
    \]
    is cofinal which implies the claim.
\end{proof}

\begin{ex}\label{ex:universal} The universal example of a cocyclic object with values in an $\infty$-category with
    colimits is the Yoneda embedding $j: \N(\bLambda) \to \P(\bLambda)$ where $\P(\bLambda)$ denotes
    the $\infty$-category $\Fun(\N(\bLambda)^{\op}, \S)$ of cyclic spaces. We obtain a functor
    \[
	\rho_j: \N(\Rib) \lra \P(\bLambda)
    \]
    which realizes a ribbon graph as a cyclic space. This functor is the universal state sum: Given
    any cocyclic object $X: \N(\bLambda) \to \C$ where $\C$ has colimits, we have 
    \begin{equation}\label{eq:kanext}
		    \rho_X \simeq j_!X \circ \rho_j
    \end{equation}
    where we use Proposition \ref{prop:pointwise} and the fact (\cite[5.1.5.5]{lurie:htt}) that $j_!X$ commutes with colimits. 
    We will use the notation
    \[
	    \Lambda({\Gamma}) := \rho_j(\Gamma) 
    \]
    for the state sum of $j$ on $\Gamma$.
\end{ex}

The following proposition explains the relevance of the $2$-Segal condition for state sums.

\begin{prop}\label{thm:funcontract} Let $\C$ be an $\infty$-category with
    	colimits and let $X: \N(\bLambda) \to \C$ be a cocyclic object. Then $X$ is unital $2$-Segal
	if and only if the state sum functor
	\[
		\rho_X: \N(\Rib) \lra \C, \; \Gamma \mapsto X(\Gamma)
	\]
	maps edge contractions in $\Rib$ to equivalences in $\C$.
\end{prop}
\begin{proof}
	Let $(f,\eta): \Gamma \to \Gamma'$ be a morphism of ribbon graphs. The associated morphism 
	$\rho_X(f,\eta): X(\Gamma) \to X(\Gamma')$ is given by the composite
	\[
	\colim (X \circ \delta) \overset{X \circ \eta^{\vee}}{\lra}  \colim (X \circ \delta' \circ f^{\op})
	\lra
	\colim (X \circ \delta').
	\]
	We claim that, if $(f,\eta)$ is an edge contraction, then $X \circ \eta^{\vee}$ exhibits $X
	\circ \delta'$ as a left Kan extension of $X \circ \delta$. This implies the result since a
	colimit is given by the left Kan extension to the final category and left Kan extension
	functors are functorial in the sense $f_{!} \circ g_{!} \simeq (f \circ g)_{!}$.
	The claim follows immediately from the argument of Proposition \ref{prop:ribkan}, Lemma
	\ref{lem:pullback}, and Remark \ref{rem:2segalpush} .
\end{proof}

\begin{rem} In the situation of Proposition \ref{thm:funcontract}, we may restrict ourselves to the
	subcategory $\Rib' \subset \Rib$ generated by egde contractions and isomorphisms. Then by
	the statement of the theorem, we obtain a functor
	\[
		\N(\Rib')_{\simeq} \to \C, \Gamma \mapsto X(\Gamma)
	\]
	where $\N(\Rib')_{\simeq} = \on{Sing}|\N(\Rib')|$ denotes the $\infty$-groupoid completion
	of the $\infty$-category $\N(\Rib')$. The automorphism group in $\N(\Rib')_{\simeq}$ of a 
	ribbon graph $\Gamma$ which represents a stable oriented marked
	surface $(S,M)$ can be identified with the mapping class group $\on{Mod}(S,M)$. The above functor
	implies the existence of an $\infty$-categorical action of $\on{Mod}(S,M)$ on $X(\Gamma)$
	which is a main result of \cite{dk:triangulated}.
\end{rem}

\subsection{Paracyclically ordered sets and framed graphs}

Let $(S,M)$ be a stable oriented marked surface. We may interpret the orientation as a reduction of
structure group of the tangent bundle of $S \setminus M$ along $\GL_2^+(\RR) \subset \GL_2(\RR)$. We
define a {\em framing} of $(S,M)$ as a further lift of the structure group along the universal cover 
\[
	\widetilde{\GL_2^+(\RR)} \lra \GL_2^+(\RR).
\]
Up to contractible choice, this datum is equivalent to a trivialization of the tangent
bundle of $S \setminus M$. In this section, we describe a state sum formalism based on a
combinatorial model for stable framed marked surfaces as developed in \cite{dk:crossed}. This
amounts to a variation of the constructions in the previous section, obtained by replacing the
cyclic category by the {\em paracyclic category}.

\subsubsection{Paracyclically ordered sets}

Let $J$ be a finite nonempty set. We define a {\em paracyclic order} on $J$ to be a cyclic order on
$J$ together with the choice of a $\ZZ$-torsor $\widetilde{J}$ and a $\ZZ$-equivariant map
$\widetilde{J} \to J$. A morphism of paracyclically ordered sets $(J, \widetilde{J}) \to (J',
\widetilde{J'})$ consists of a commutative diagram of sets
\[
	\xymatrix{
		\widetilde{J} \ar[r]^{\widetilde{f}}\ar[d] & \widetilde{J'} \ar[d]\\
		J \ar[r]^{f} & J'
	}
\]
such that $\widetilde{f}$ is monotone with respect to the $\ZZ$-torsor linear orders. The lift
$\widetilde{f}$ equips $f$ naturally with the structure of a morphism of cyclically ordered sets so
that we obtain a forgetful functor
\[
	\bLambda_{\infty} \lra \bLambda
\]
where $\bLambda_{\infty}$ denotes the category of paracyclically ordered sets. As for cyclically
ordered sets, there is a skeleton $\Lambda_{\infty} \subset \bLambda_{\infty}$ consisting of
standard paracyclically ordered sets $(\cn, \widetilde{\cn})$ where we define $\widetilde{\cn} =
\ZZ$ and $\widetilde{\cn} \to \cn$ is given by the natural quotient map.

Given a paracyclically ordered set $(J, \widetilde{J})$, then the cyclic order on the interstice dual
$J^{\vee}$ lifts to a natural paracyclic order given by 
\[
	\widetilde{J^{\vee}} = 
	\Hom_{\bLambda_{\infty}}( (J,\widetilde{J}),(\langle 0 \rangle, \widetilde{ \langle 0 \rangle} )).
\]
This construction extends the self duality of $\bLambda$ to one for $\bLambda_{\infty}$.
All statements of Section \ref{sec:cyclically} hold mutatis mutandis for $\bLambda_{\infty}$. 

\subsubsection{Framed graphs and state sums}

We define a {\em framed graph} $\Gamma$ to be a graph $\Gamma$, equipped with a lift
\[
\begin{tikzcd}
	& \bLambda_{\infty} \arrow{d}\\
	I(\Gamma) \arrow{r} \arrow[dashed]{ur} & \Set
\end{tikzcd}
\]
of the incidence diagram of $\Gamma$ where $\bLambda_{\infty}$ denotes the category of paracyclically
ordered sets. Framed graphs form a category $\Rib_{\infty}$ which is defined in complete analogy
with $\Rib$.

Let $\Gamma$ be a framed graph with incidence diagram 
	\[
		\gamma: I(\Gamma) \to \bLambda_{\infty}.
	\]
	Let $e = \{h, \tau(h)\}$ be an edge in $\Gamma$ incident to the vertices $v = s(h)$ and $w =
	s(\tau(h))$, and let $h'$ be a half-edge incident to $w$. Let $\widetilde{h}$ be a lift of
	$h$ to an element of the $\ZZ$-torsor $\widetilde{H(v)}$ which is part of the paracyclic
	structure on $H(v)$. Then we may transport this lift along
	the edge $e$ to obtain a lift of $h'$ to an element $\widetilde{h'}$ of $\widetilde{H(w')}$
	as follows: There is a unique lift of $\widetilde{\tau(h)} \in \widetilde{H(w')}$ of
	$\tau(h)$ which maps to $\widetilde{h} - 1$ under $\gamma(\tau(h))$. We the set 
	$\widetilde{h'} = \widetilde{\tau(h)} + i$ where $i\ge 0$ is minimal such that
	$\widetilde{h'}$ lifts $h'$.
	Iterating this transport along a loop $l$ returning to the half-edge $h$, we obtain another lift
	$\overline{h}$ of $h$. The integer $\overline{h} -
	\widetilde{h}$ only depends on $l$ and we refer to it as the {\em winding
	number of $l$}.

\begin{rem} As explained in \cite{dk:crossed}, framed graphs provide a combinatorial model
	for stable marked surfaces $(S,M)$ equipped with a trivialization of the tangent bundle of
	$S \setminus M$. The above combinatorial construction then coincides with the geometric
	concept of winding number computed with respect to the framing.
\end{rem}

Let $\C$ be an $\infty$-category with colimits, and let $X^{\bullet}: \N(\bLambda_{\infty}) \to \C$ be a coparacyclic object in
$\C$. Then, given a framed graph $\Gamma$, we have a state sum 
\[
	X(\Gamma) = \colim X \circ \delta
\]
where $\delta: I(\Gamma)^{\op} \to \bLambda_{\infty}$ denotes the coincidence diagram of $\Gamma$.
The various state sums naturally organize into a functor
\[
	\N(\Rib_{\infty}) \lra \C, \Gamma \mapsto X(\Gamma).
\]

\begin{rem} As shown in \cite{dk:crossed}, the state sum $X(\Gamma)$ of a framed graph with values in a
	coparacyclic unital $2$-Segal object $X$ admits an action of the {\em framed} mapping class group
	of the surface.
\end{rem}

\subsection{The universal loop space}
\label{sec:universal}

We give a first example of a state sum which will play an important role later on. 
Consider the functor
\[
		\Lambda \lra \Grp,\; \cn \mapsto \pi_1(D/\{0,1,\dots,n\})
\]
where $D/\{0,1,\dots,n\}$ denotes the quotient of the unit disk by $n+1$ marked points on
the boundary. Replacing the fundamental groups by their nerves, we obtain a functor
\begin{equation}\label{eq:uniloop}
		L^{\bullet}: \N(\Lambda) \lra \S_{\ast}
\end{equation}
where $\S_{\ast}$ denotes the $\infty$-category of pointed spaces. The
cosimplicial pointed space underlying $L^{\bullet}$ is $1$-Segal and hence, by Proposition
\ref{prop:12segal}, unital $2$-Segal.

\begin{rem}
	Note that the group $\pi_1(D/\{0,1,\dots,n\})$ is a free group on $n$ generators so that
	$L^n$ is equivalent to a bouquet of $n$ one-dimensional spheres. Given any pointed space
	$X$, the cyclic pointed $1$-Segal space $\Map(L^{\bullet}, X)$ describes the loop space $\Omega X$
	together with its natural group structure. Similarly, the cocyclic pointed $1$-Segal space
	$L^{\bullet} \otimes X$ describes the suspension of $X$ together with its natural cogroup structure.
\end{rem}

\begin{prop}\label{prop:universal} Let $(S,M)$ be a stable marked oriented surface represented by a ribbon graph $\Gamma$.
	Then we have 
	\[
		L(\Gamma) \simeq S/M.
	\]
\end{prop}
\begin{proof} We can compute the state sum defining $L(\Gamma)$ explicitly as a homotopy colimit in
	the category of pointed spaces. We may replace the diagram by a homotopy equivalent one which
	assigns to each $n$-corolla of the graph $\Gamma$ the space $D/\{0,1,\dots,n\}$. The
	homotopy colimit is then obtained by identifying the boundary cycles of the various spaces 
	$D/\{0,1,\dots,n\}$ according to the incidence relations given by the edges of the ribbon
	graphs. It is apparent that the resulting space is equivalent to the quotient space $S/M$.
\end{proof}

For a slight elaboration on Proposition \ref{prop:universal}, let $\C$ be a stable $\infty$-category
with colimits and let $E$ be an object of $\C$. 
By \cite[1.4.2.21]{lurie:algebra}, the functor 
\[
	\Sp(\C) = \Fun^c(\S_{\ast}^{\on{fin}}, \C) \to \C, f \mapsto f(S^0)
\]
from the $\infty$-category of spectrum objects in $\C$ to $\C$ is an equivalence. Therefore, the
object $E$ defines an essentially unique functor 
\begin{equation}\label{eq:E}
		E: \S_{\ast}^{\on{fin}} \to \C 
\end{equation}
which we still denote by $E$. 

\begin{defi} Let $(X,Y)$ be a pair of finite spaces. We introduce the pointed quotient space $X/Y$
	as the pushout
	\[
		\xymatrix{ Y \ar[r]\ar[d] & X\ar[d]\\
		\ast \ar[r] & X/Y }
	\]
	in $\S$. The object $E(X/Y)$ of $\C$ is called the {\em relative homology of the pair
	$(X,Y)$ with coefficients in $E$}. In the case when $\C$ is the category of spectra, this
	terminology agrees with the customary one.
\end{defi}

We introduce the cocyclic object
\begin{equation}\label{eq:universal}
		L_E = E(L^{\bullet}) : \N(\Lambda) \lra \C
\end{equation}
obtained from \eqref{eq:uniloop} by postcomposing with $E: \S_{\ast}^{\on{fin}} \to \C$.

\begin{prop}\label{prop:spectrum} 
	Let $(S,M)$ be a stable marked oriented surface represented by a ribbon graph $\Gamma$.
	Then we have an equivalence
	\[
		L_E(\Gamma) \simeq E(S/M)
	\]
	in $\C$.
\end{prop}
\begin{proof}
The functor $E(-)$ commutes with finite colimits so that the statement follows
immediately from Proposition \ref{prop:universal}.
\end{proof}

\begin{rem}\label{rem:universal} 
	We may pullback a cocyclic $2$-Segal object $X^{\bullet}$ along the functor
	$\Lambda_{\infty} \to \Lambda$ to obtain a coparacyclic $2$-Segal object
	$\widetilde{X^{\bullet}}$. Given a framed graph $\Gamma$, we have 
	\[
		\widetilde{X}(\Gamma) \simeq X(\overline{\Gamma})
	\]
	where $\overline{\Gamma}$ is the ribbon graph underlying $\Gamma$. In particular, in the
	context of Proposition \ref{prop:spectrum}, we obtain  
	\[
		\widetilde{L_E}(\Gamma) \simeq E(S/M)
	\]
\end{rem}

\section{Differential graded categories}
\label{sec:dgcat}

\subsection{Morita equivalences}

We introduce some terminology for the derived Morita theory of differential graded categories and refer the reader to
\cite{tabuada:model,toen:morita} for more detailed treatments.
Let $k$ be a commutative ring, and let $\dgcat$ be the category of small $k$-linear differential
$\ZZ$-graded categories. Recall that a functor $f: A \to B$ is called a {\em quasi-equivalence} if
\begin{enumerate}
	\item the functor $H^0(f): H^0(A) \to H^0(B)$ of homotopy categories is an equivalence of
		categories,
	\item for every pair of objects $(x,y)$ in $A$, the morphism $f: \Hom_{A}(x,y) \to
		\Hom_{B}(f(x),f(y))$ of complexes is a quasi-isomorphism.
\end{enumerate}
We denote by $\Lqe$ the $\infty$-category obtained by localizing $\dgcat$
along quasi-equivalences (\cite[1.3.4.1]{lurie:algebra}). The collection of quasi-equivalences can be supplemented to a
combinatorial model structure on $\dgcat$ which facilitates calculations in $\Lqe$

Given dg categories $A$,$B$, we denote the dg category of enriched functors from $A$ to $B$ by
$\UHom(A,B)$. We denote by $\Mod_k$ the dg category of unbounded complexes of $k$-modules, and
further, by $\Mod_A$ the dg category $\UHom(A^{\op}, \Mod_k)$. We equip $\Mod_A$ with the projective
model structure and denote by $\Perf_A \subset \Mod_A$
the full dg category spanned by those objects $x$ such that
\begin{enumerate}
	\item $x$ is cofibrant,
	\item the image of $x$ in $H^0(\Mod_A)$ is compact, i.e., $\Hom(x,-)$ commutes with coproducts. 
\end{enumerate}
Given a dg functor $f: A \to B$, we have a Quillen adjunction 
\[
	f_!: \Mod_A \lra \Mod_B : f^*
\]
and obtain an induced functor
\begin{equation}\label{eq:perf}
	f_!: \Perf_A \lra \Perf_B.
\end{equation}
The functor $f: A \to B$ is called a {\em Morita equivalence} if the induced functor \eqref{eq:perf}
is a quasi-equivalence. We denote by $\Lmo$ the $\infty$-category obtained by localizing $\dgcat$
along Morita equivalences. We have an adjunction
\begin{equation}\label{eq:lmo}
		l: \Lqe \longleftrightarrow \Lmo: i
\end{equation}
where $i$ is fully-faithful so that $l$ is a localization functor. 

Let $\dgcatt$ denote the category of small $\k$-linear differential $\Zt$-graded categories. All of
the above theory can be translated mutatis mutandis via the adjunction
\[
	P: \dgcat \longleftrightarrow \dgcatt: Q
\]
which is a Quillen adjunction with respect to an adaptation of the quasi-equivalence model structure
on $\dgcatt$. The periodization functor $P$ associates to a differential $\ZZ$-graded category the
differential $\Zt$-graded category with the same objects and $\Zt$-graded mapping complexes obtained by
summing over all even (resp. odd) terms of the $\ZZ$-graded mapping complexes. We will refer to the $\Zt$-graded analogues 
of the above constructions via the superscript $(2)$.

\begin{rem}\label{rem:colimits} Note that, due to the adjunction \eqref{eq:lmo}, the functor $l$ commutes with colimits
	so that we may compute colimits in the category $\Lmo$ as colimits in $\Lqe$. The latter
	category is equipped with the quasi-equivalence model structure so that, by
	\cite{lurie:htt}, we can compute colimits as homotopy colimits with respect to this model
	structure. The analogous statement holds for the $\Zt$-graded variants.
\end{rem}

\subsection{Exact sequences of dg categories}
\label{sec:exact}

A morphism in $\Lmo$ is called quasi-fully faithful if it is equivalent to the image of a quasi-fully 
faithful morphism under the localization functor $\N(\dgcat) \to \Lmo$.
A pushout square 
\[
	\begin{tikzcd}
		S \arrow[hookrightarrow]{r}{g} \arrow{d} & T \arrow[two heads]{d}{f} \\
	0 \arrow{r} & U 
	\end{tikzcd}
\]
in $\Lmo$ with $g$ quasi-fully faithful is called an {\em exact sequence}. The following technical
statement will be used below:

\begin{lem}\label{lem:qffpushout} Quasi-fully faithful morphisms are stable under pushouts in $\Lmo$.
\end{lem}
\begin{proof}
	The adjunction \eqref{eq:lmo} implies that the left adjoint $l: \Lqe \to \Lmo$ preserves
	colimits so that it suffices to prove the corresponding statement for $\Lqe$. To this end it
	suffices to show that quasi-fully faithful functors in the category $\dgcat$ are stable 
	under homotopy pushouts with respect to the quasi-equivalence model structure defined in
	\cite{tabuada:model}. Given
	a diagram 
	\begin{equation}\label{eq:pushoutdgcat}
			\xymatrix{
				S \ar[r]^g \ar[d]^f &  T\\
				S'
			}
	\end{equation}
	with $g$ quasi-fully faithful, we may assume that all objects are cofibrant, and $f$,$g$ are
	cofibrations so that the homotopy pushout is given by an ordinary pushout. Denoting by $I$
	the set of generating cofibrations of $\dgcat$, we may, by Quillen's small object argument,
	factor the morphism $f$ as
	\[
		S \overset{f_1}{\to} \widetilde{S'} \overset{f_2}{\to} S'
	\]
	where $f_1$ is a relative $I$-cell complex and $f_2$ is a trivial fibration. Forming
	pushouts we obtain a diagram
	\begin{equation}
			\xymatrix{
				S \ar[r]^g \ar[d]^{f_1} &  T \ar[d] \\
				\widetilde{S'} \ar[r]^{\widetilde{g}} \ar[d]^{f_2} & T'\ar[d]^{r} \\
				S' \ar[r]^{g'} & T''.
			}
	\end{equation}
	Since $\widetilde{g}$ is a cofibration, the bottom square is a homotopy pushout square and
	thus $r$ is a quasi-equivalence. It hence suffices to show that $\widetilde{g}$ is
	quasi-fullyfaithful so that we may assume that $f$ is a relative $I$-cell complex. Using that filtered colimits of
	complexes are homotopy colimits (\cite{toen-vaquie:moduli}), and hence preserve quasi-isomorphisms, we reduce to the
	case that $f$ is the pushout of a single generating cofibration from $I$. This leaves us with two cases:
	\begin{enumerate}
		\item $S'$ is obtained from $S$ by adjoining one object, and the pushout $T'$ of
			\eqref{eq:pushoutdgcat}
			is obtained from $T$ by adjoining one object. Clearly, the functor $S' \to T'$ is quasi-fully faithful.
		\item $S'$ is obtained from $S$ by freely adjoining a morphism $p: a \to b$ of some degree
			$n$ between objects $a,b$ of $S$ where $d(p)$ is a prescribed morphism $q$ of $S$. 
			The morphism complex between objects $x,y$ in $S'$ can be described
			explicitly as 
			\[
				S'(x,y)  = \bigoplus_{n \ge 0} S(x,a) \otimes k p  \otimes S(b,a)
				\otimes k p \otimes \dots \otimes S(b,y)
			\]
			where $n$ copies of $k p$ appear in the $n$th summand. The differential is 
			given by the Leibniz rule where, upon replacing $p$ by $d(p) = q$, we also
			compose with the neighboring morphisms so that the level is decreased from
			$n$ to $n-1$. The morphism complexes of the pushout $T'$ have a admit
			the analogous expression with $p$ replaced by $g(p)$. We have to show that,
			for every pair of objects, the morphism of
			complexes
			\[
				S'(x,y) \to T'(g(x),g(y))
			\]
			is a quasi-isomorphism. To this end, we filter both complexes by 
			the level $n$. On the associated graded complexes we have a
			quasi-isomorphism, since $g$ is quasi-fully faithful. The corresponding
			spectral sequence converges which yields the desired quasi-isomorphism.
	\end{enumerate}
\end{proof}

\begin{rem} The proof of the lemma works verbatim for $\dgcatt$ instead of $\dgcat$.
\end{rem}

\section{Topological Fukaya categories}

\subsection{$\Zt$-graded}
\label{sec:zt}

Let $k$ be a commutative ring, and let $R  = k[z]$ denote the polynomial ring with coefficients in
$k$, considered as a $\ZZ/(n+1)$-graded $k$-algebra with $|z| = 1$. A matrix factorization $(X, d_X)$ of $w =
z^{n+1}$ consists of
\begin{itemize}
	\item a pair $X^0$, $X^1$ of $\ZZ/(n+1)$-graded $R$-modules,
	\item a pair of homogeneous $R$-linear homomorphisms 
		\[
		\xymatrix{
		X^0 \ar@/_1ex/[r]_{d^0} & X^1 \ar@/_1ex/[l]_{d^1}
		}
		\]
		of degree $0$,
\end{itemize}
such that
\begin{itemize}
	\item $d^1 \circ d^0 = w \id_{X^0}$ and $d^0 \circ d^1 = w \id_{X^1}$.
\end{itemize}

\begin{exa}\label{exa:scalar} 
	For $i,j \in \ZZ/(n+1)$, $i \ne j$, we have a corresponding {\em scalar} matrix factorization
	$[i,j]$ defined as
	\[
		\xymatrix{
		k[z](i) \ar@/_1ex/[r]_{z^{j-i}} & k[z](j) \ar@/_1ex/[l]_{z^{i-j}}
	}
	\]
	where the exponents of $z$ are to be interpreted via their representatives in
	$\{1,2,\dots,n\}$. For $i = j$, we have two scalar matrix factorizations
	\[
		\xymatrix{
		k[z](i) \ar@/_1ex/[r]_{1} & k[z](i) \ar@/_1ex/[l]_{z^{n+1}}
	}
	\]
	and
	\[
		\xymatrix{
		k[z](i) \ar@/_1ex/[r]_{z^{n+1}} & k[z](i) \ar@/_1ex/[l]_{1}
	}
	\]
	which we denote by $[i,i]_r$ and $[i,i]_l$, respectively.
\end{exa}

Given matrix factorizations $X$,$Y$ of $w$, we form the $\Zt$-graded $k$-module
$\Hom^{\bullet}(X,Y)$ with
\begin{align*}
	\Hom^{0}(X.Y) & = \Hom_R(X^0,Y^0) \oplus \Hom_R(X^1,Y^1)\\
	\Hom^{1}(X.Y) & = \Hom_R(X^0,Y^1) \oplus \Hom_R(X^1,Y^0)
\end{align*}
where $\Hom_R$ denotes homogeneous $R$-linear homomorphisms of degree $0$.
It is readily verified that the formula
\[
	d(f) = d_Y \circ f - (-1)^{|f|} f \circ d_X
\]
defines a differential on $\Hom^{\bullet}(X,Y)$, i.e., $d^2 = 0$. Therefore, the
collection of all matrix factorizations of $w$ organizes into a differential $\Zt$-graded 
$k$-linear category which we denote by $\MF^{\ZZ/(n+1)}(k[z],z^{n+1})$.
We further define
\[
	\Ft^n \subset \MF^{\ZZ/(n+1)}(k[z],z^{n+1})
\]
to be the full dg subcategory spanned by the scalar matrix factorizations from Example
\ref{exa:scalar}.

\begin{thm}\label{thm:matrix} The association $n \mapsto \Ft^n$ extends to a cocyclic object
	\[
		\Ft: \N(\Lambda) \to \Lmot
	\]
	which is unital $2$-Segal.
\end{thm}
\begin{proof}\cite{dk:triangulated}
\end{proof}

\begin{rem}\label{rem:an} Let $A^n$ denote the $k$-linear envelope of the category associated to the linearly
	ordered set $\{1,2,\dots,n\}$, considered as a differential $\Zt$-graded category
	concentrated in degree $0$. There is a dg functor
	\[
		g: A^n \to \Ft^n, i \mapsto [0,i]
	\]
	which maps the generating morphism $i \to j$ to the closed morphism of matrix factorizations 
	\[
		\xymatrix{
			k[z] \ar@/_1ex/[r]_{z^i}\ar[d]_{1} & k[z](i) \ar[d]^{z^{j-i}}\ar@/_1ex/[l]_{z^{n+1-i}}\\
			k[z] \ar@/_1ex/[r]_{z^j} & k[z](j) \ar@/_1ex/[l]_{z^{n+1-j}}.
		}
	\]
	An explicit calculation shows:
	\begin{itemize}
		\item The functor $g$ is quasi-fully faithful. 
		\item The object $[i,j]$ is a cone over the above morphism $[0,i] \to [0,j]$.
		\item The objects $[i,i]_l$ and $[i,i]_r$ are zero objects.
	\end{itemize}
	These observations imply that the functor $g$ is a Morita equivalence. The reason for using
	$\Ft^n$ instead of the much simpler dg category $A^n$ is the following: the cocyclic object in Theorem
	\ref{thm:matrix} is difficult to describe in terms of $A^n$ while the association $n \mapsto
	\Ft^n$ defines a strict functor $\Lambda \to \dgcatt$ which induces $\Ft^{\bullet}$ by passing
	to the Morita localization.
\end{rem}

The state sum formalism of Section \ref{sec:statesum} yields a functor
\[
	\rho_{\Ft}: \N(\Rib) \lra \Lmot,\; \Gamma \mapsto \Ft(\Gamma).
\]
The state sum $\Ft(\Gamma)$ of $\Ft$ on a ribbon graph $\Gamma$ is called 
the ($\Zt$-graded) {\em topological Fukaya category} of $\Gamma$.

\begin{ex}\label{exa:kt}
	Consider the ribbon graph $\Gamma$ given by 
	
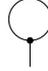
\[
		\tikz[scale=0.7,baseline=(current bounding box.center) ]{
			  \filldraw[black] (0,0) circle(1.5pt);
			  \draw (0,.4) circle(0.4);
			  \draw (0,0) -- (0,-.5); 
		}
	\]
	The corresponding topological Fukaya category $\Ft(\Gamma)$ 
	can, by Proposition \ref{prop:ribpushouts}(1), Remark \ref{rem:colimits}, and Remark \ref{rem:an}, 
	be computed as the homotopy pushout of the diagram 
	\[
		\xymatrix{
			A^0 \amalg A^0 \ar[r]\ar[d] & A^0 \\
			A^1  & 
		}
	\]
	with respect to the quasi-equivalence model structure on $\dgcatt$. All objects
	are cofibrant and the vertical functor is a cofibration so that the homotopy pushout can be
	computed as an ordinary pushout. 
	Therefore, we obtain 
	\[
		\Ft(\Gamma) \simeq k[t] 
	\]
	the $k$-linear category with one object and endomorphism ring $k[t]$, considered as a
	differential $\Zt$-graded category with zero differential. We can therefore interpret
	$\Ft(\Gamma)$ as the $\Zt$-folding of the bounded derived dg category of coherent sheaves on the
	affine line $\AA^1_k$ over $k$.
\end{ex}

\begin{ex}\label{ex:kronecker}
	Consider the ribbon graph $\Gamma$ given by 
	
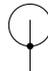
\[
		\tikz[yshift=4ex,scale=0.7]{
			  \filldraw[black] (270:.4) circle(1.5pt);
			  \draw (0,0) circle(0.4);
			  \draw (270:.4) -- (270:.9);
			  \draw (90:.1) -- (270:.4);
		}
	\]
	We replace $\Gamma$ by the ribbon graph $\Gamma'$
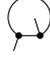
	\[
		\tikz[yshift=4ex, rotate=90, scale=0.7]{
			\draw (-150:.4) arc (-150:150:.4);
			\draw (-150:.4) -- (150:.4);
			\draw (-150:.4) -- (-100:.1);
			\draw (150:.4) -- (155:.7);
			\filldraw[black] (-150:.4) circle(1.5pt);
			\filldraw[black] (150:.4) circle(1.5pt);
		}
	\]
	which, by Proposition \ref{thm:funcontract}, has an equivalent topological Fukaya category.
	Using Proposition 4.2.3.8 in \cite{lurie:htt}, Remark \ref{rem:colimits}, and Remark
	\ref{rem:an}, we can obtain $\Ft(\Gamma')$ as the homotopy pushout of 
	\[
		\xymatrix{
			A^0 \amalg A^0 \ar[r]\ar[d] & A^1 \\
			A^1  & 
		}
	\]
	with respect to the quasi-equivalence model structure on $\dgcatt$. Since all objects in
	this diagram are cofibrant and all functors cofibrations, we can form the ordinary pushout
	to obtain a description of $\Ft(\Gamma)$ as the $k$-linear category generated by the Kronecker
	quiver with two vertices $0$ and $1$ and two edges from $0$ to $1$.
	We can therefore interpret
	$\Ft(\Gamma)$ as the $\Zt$-folding of the bounded derived dg category of coherent sheaves on the
	projective line $\mathbb P^1_k$ over $k$.
\end{ex}

\subsection{$\ZZ$-graded}

We discuss a $\ZZ$-graded variant of the topological Fukaya category which can be associated to any
framed stable marked surface and provides a lift of the $\Zt$-graded category associated to the
underlying oriented surface. It can be obtained via a minor modification of the constructions in Section
\ref{sec:zt}: we introduce a differential $\ZZ$-graded category $\MF^{\ZZ}(k[z], z^{n+1})$ of
$\ZZ$-graded matrix factorizations.

Let $k$ be a commutative ring, and let $R = k[z]$ denote the polynomial ring, considered as a
$\ZZ$-graded ring with $|z| = 1$. A $\ZZ$-graded matrix factorization $(X,d)$ of $w = z^{n+1}$
consists of
\begin{itemize}
	\item a pair $X^0$, $X^1$ of $\ZZ$-graded $R$-modules,
	\item a pair of homogeneous $R$-linear homomorphisms 
		\[
		\xymatrix{
		X^0 \ar@/_1ex/[r]_{d^0} & X^1 \ar@/_1ex/[l]_{d^1}
		}
		\]
		where $|d^0| = 0$ and $|d^1| = n+1$.
\end{itemize}
such that
\begin{itemize}
	\item $d^1 \circ d^0 = w \id_{X^0}$ and $d^0 \circ d^1 = w \id_{X^1}$.
\end{itemize}

\begin{exa}\label{exa:scalarZ} 
	For $i, j \in \ZZ$, $0 \le j - i \le n+1$, we have a corresponding {\em scalar} matrix factorization
	$[i,j]$ defined as
	\[
		\xymatrix{
		k[z](i) \ar@/_1ex/[r]_{z^{j-i}} & k[z](j) \ar@/_1ex/[l]_{z^{i-j+n+1}}.
	}
	\]
\end{exa}

Given matrix factorizations $X,Y$, we define a $\ZZ$-graded mapping complex $\Hom^{\bullet}(X,Y)$ as
follows. We extend $X$ to a $\ZZ$-sequence 
\[
	\dots \overset{d}{\lra} \widetilde{X}^{i-1} \overset{d}{\lra} \widetilde{X}^i  \overset{d}{\lra}  \widetilde{X}^{i+1} \lra \dots
\]
setting 
\[
	\widetilde{X}^i :=  X^{\overline{i}}( (n+1) \left \lfloor{\frac{i}{2}}\right \rfloor)
\]
where $\overline{i}$ denotes the residue of $i$ modulo $2$. In particular, we have $\widetilde{X}^{i+2} =
\widetilde{X}^i(n+1)$. The morphisms $d$ in the sequence
$\widetilde{X}$ are homogeneous of degree $0$ and satisfy $d^2 = w$. Given matrix factorizations
$X,Y$, we define the $\ZZ$-graded complex $\Hom^{\bullet}(X,Y)$ by setting
\[
	\Hom^j(X,Y) = \{(f_i)_{i \in \ZZ} \; |\; f_{i+2} = f_i(n+1)\} \subset \prod_{i \in \ZZ} \Hom_R(\widetilde{X}^i , \widetilde{X}^{i+j}) 
\]
equipped with the differential given by the formula
\[
	d(f) = d_{\widetilde{Y}} f - (-1)^{|f|} f \circ d_{\widetilde{X}}.
\]
Therefore, the collection of all $\ZZ$-graded matrix factorixations of $w$ organizes into a
differential $\ZZ$-graded 
$k$-linear category which we denote by $\MF^{\ZZ}(k[z],z^{n+1})$.
In analogy with the $\Zt$-graded case, we further define
\[
	\Ff^n \subset \MF^{\ZZ}(k[z],z^{n+1})
\]
to be the full dg subcategory spanned by the scalar matrix factorizations from Example
\ref{exa:scalarZ}.

\begin{thm}\label{thm:matrixZ} The association $n \mapsto \Ff^n$ extends to a coparacyclic object
	\[
		\Ff: \N(\Lambda_{\infty}) \to \Lmo
	\]
	which is unital $2$-Segal.
\end{thm}
\begin{proof}\cite{lurie:rotation,dk:crossed}
\end{proof}

\begin{rem} \label{rem:anZ} The statement of Remark \ref{rem:an} has a $\ZZ$-graded analog: For ever
	$n \ge 0$, the association $i \mapsto [0,i]$ defines a Morita equivalence of $\ZZ$-graded categories
	\[
		A^n_{\ZZ} \lra \Ff^n
	\]
	where $A^n_{\ZZ}$ denotes the $\ZZ$-graded variant of $A^n$.
\end{rem}

We obtain a state sum functor
\[
	\rho_{\Ff}: \N(\Rib_{\infty}) \lra \Lmo,\; \Gamma \mapsto \Ff(\Gamma).
\]
The state sum $\Ff(\Gamma)$ of $\Ff$ on a framed graph $\Gamma$ is called 
the ($\ZZ$-graded) {\em topological Fukaya category} of $\Gamma$.

\begin{rem} It is immediate from the definitions that we have a commutative square
	\[
		\xymatrix{
			\N(\Lambda_{\infty}) \ar[r]^{\Ff}\ar[d] & \Lmo \ar[d]^P\\
			\N(\Lambda) \ar[r]^{\Ft} & \Lmot
		}
	\]
	in $\Cat_{\infty}$ where the left vertical arrow is the natural forgetful functor.
	Since the periodization functor $P$ commutes with colimits it follows that, for a framed graph $\Gamma$, we have
	\[
		P(\Ff(\Gamma)) \simeq \Ft(\bar{\Gamma})
	\]
	where $\bar{\Gamma}$ denotes the ribbon graph underlying $\Gamma$. In other words, $\Ff(\Gamma)$ 
	provides a lift of the differential $\Zt$-graded category $\Ft(\bar{\Gamma})$ to a differential
	$\ZZ$-graded category.
\end{rem}

\begin{ex}\label{exa:ktZ}
	Consider the ribbon graph $\Gamma$ given by 
	
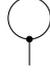
\[
		\tikz[scale=0.7,baseline=(current bounding box.center) ]{
			  \filldraw[black] (0,0) circle(1.5pt);
			  \draw (0,.4) circle(0.4);
			  \draw (0,0) -- (0,-.5); 
		}
	\]
	equipped with framing corresponding to the winding number $r$ 
	around the loop. An analogous calculation to Example \ref{exa:kt} yields 
	\[
		\Ff(\Gamma) \simeq k[t] 
	\]
	the $k$-linear category with one object and endomorphism ring $k[t]$, $|t| = 2r$, considered as a
	differential $\ZZ$-graded category with zero differential. The endomorphism dga of the
	$k[t]$-module $k \cong k[t]/(t)$ is given by the graded ring $k[\delta]$, $\delta^2 =0$,
	$|\delta| = 1 - 2r$.
\end{ex}

\begin{ex}\label{ex:kroneckerZ}
	Consider the ribbon graph $\Gamma$
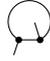
	\[
		\tikz[yshift=3ex, rotate=90,scale=0.7]{
			\draw (-150:.4) arc (-150:150:.4);
			\draw (-150:.4) -- (150:.4);
			\draw (-150:.4) -- (-100:.1);
			\draw (150:.4) -- (155:.7);
			\filldraw[black] (-150:.4) circle(1.5pt);
			\filldraw[black] (150:.4) circle(1.5pt);
		} 
	\]
	with framing corresponding to the winding number $r$.
	An analogous calculation to Example \ref{ex:kronecker} gives a description of
	$\Ff(\Gamma)$ as the $k$-linear category generated by the {\em graded} Kronecker
	quiver with two vertices $0$ and $1$ and two edges from $0$ to $1$ where the edges have
	degree $0$ and $2r$, respectively. Therefore, $\Ff(\Gamma)$ can be interpreted as a
	twisted version of the derived dg category of ${\mathbb P}^1_k$ (cf. \cite{seidel:exact} for an
	appearance of this quiver in another Fukaya-categorical context). The objects of
	$\Perf_{\Ff(\Gamma)}$ given by the cones of the two edges of the quiver have endomorphism 
	algebras given by $k[\delta]$, $\delta^2 = 0$ where $\delta$ has degree $1- 2r$ and $1 + 2r$,
	respectively.
\end{ex}

\section{Localization and Mayer-Vietoris}

\subsection{Localization for topological Fukaya categories}
\label{sec:thomason}

We construct localization sequences for topological Fukaya categories which are analogous to the proto-localization sequences of
Thomason-Trobaugh for derived categories of schemes. A refined analysis of certain such sequences
will feature in our proof of the Mayer-Vietoris theorem in Section \ref{sec:motivic}. We focus on
the $\ZZ$-graded case, the $\Zt$-graded case can be treated mutatis mutandis.

Let $\Gamma$ be a graph. A subgraph $\Gamma' \subset \Gamma$ is called {\em open} if, for every
vertex $v \in \Gamma'$, the graph $\Gamma'$ contains all halfedges incident to $v$ in $\Gamma$. The
{\em complement} $\Gamma'' = \Gamma \setminus \Gamma'$ of an open graph $\Gamma' \subset \Gamma$ is the subgraph
of $\Gamma$ with set of vertices $V \setminus V'$ and set of halfedges $H \setminus H'$. Note that
the complement of an open graph is open. Given an open subgraph $\Gamma' \subset \Gamma$, we
define the {\em closure} $\overline{\Gamma'}$ to be the graph which is obtained from $\Gamma'$ by adding, for every external
half-edge $h$ of $\Gamma'$ which becomes internal in $\Gamma$, a new half-edge $\tau(h)$ and vertex
$v$ which is declared incident to $\tau(h)$. We further define the {\em retract} $\underline{\Gamma'}$ of
$\Gamma'$ to be the graph obtained from $\Gamma'$ by removing all half-edges which become internal
in $\Gamma$.

\begin{exa}
	Consider the graph $\Gamma$ depicted by 
	
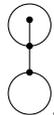
\[
		\tikz[yshift=7ex,scale=0.7]{
			  \filldraw[black] (270:.4) circle(1.5pt);
			  \filldraw[black] (270:.9) circle(1.5pt);
			  \filldraw[black] (90:.1) circle(1.5pt);
			  \draw (0,0) circle(0.4);
			  \draw (270:.4) -- (270:.9);
			  \draw (90:.1) -- (270:.4);
			  \draw (0,-1.3) circle(0.4);
		  }.
	\]
	It contains the open subgraph $\Gamma'$ given by 
	
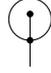
\[
		\tikz[yshift=5ex,scale=0.7]{
			  \filldraw[black] (270:.4) circle(1.5pt);
			  \filldraw[black] (90:.1) circle(1.5pt);
			  \draw (0,0) circle(0.4);
			  \draw (270:.4) -- (270:.9);
			  \draw (90:.1) -- (270:.4);
		  }
	\]
	with closure $\overline{\Gamma'}$
	
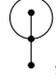
\[
		\tikz[yshift=5ex,scale=0.7]{
			  \filldraw[black] (270:.4) circle(1.5pt);
			  \filldraw[black] (90:.1) circle(1.5pt);
			  \filldraw[black] (270:.9) circle(1.5pt);
			  \draw (0,0) circle(0.4);
			  \draw (270:.4) -- (270:.9);
			  \draw (90:.1) -- (270:.4);
		  },
	\]
	retract $\underline{\Gamma'}$
	
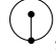
\[
		\tikz[yshift=3ex,scale=0.7]{
			  \filldraw[black] (270:.4) circle(1.5pt);
			  \filldraw[black] (90:.1) circle(1.5pt);
			  \draw (0,0) circle(0.4);
			  \draw (90:.1) -- (270:.4);
		  },
	\]
	and complement $\Gamma \setminus \Gamma'$
	
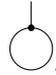
\[
		\tikz[yshift=7ex,scale=0.7]{
			  \filldraw[black] (270:.9) circle(1.5pt);
			  \draw (270:.4) -- (270:.9);
			  \draw (0,-1.3) circle(0.4);
		  }.
	\]
\end{exa}

Note that if $\Gamma$ is a framed (resp. ribbon) graph then any open subgraph inherits a canonical
framing (resp. ribbon structure). We will need the following special case of a general descent statement
which can be proved with the same technique.

\begin{prop}\label{prop:descent} Let $\C$ be an $\infty$-category with finite colimits. Let $\Gamma$
	be a graph, and let $\Gamma' \subset \Gamma$ be an open subgraph of $\Gamma$ with complement
	$\Gamma''$.  Assume that $\Gamma$ carries a ribbon (resp. framed) structure and $X$ is a
	co(para)cyclic object in $\C$. Then there is a canonical pushout square
	\begin{equation}
			\xymatrix{
				X(\Xi) \ar[r]\ar[d] & X(\Gamma'') \ar[d]\\
				X(\Gamma') \ar[r] & X(\Gamma)
			}
	\end{equation}
	in $\C$ where $\Xi$ denotes the graph given by the disjoint union of copies of the corolla
	\[
		\tikz[baseline=(current bounding box.center),scale=0.7]{
			\filldraw[black] (-.2,0) circle(1.5pt);
			\draw (-.2,-.5) -- (-.2,0) -- (-.2,.5);
		}
	\]
	indexed by those half-edges in $\Gamma'$ which become internal in $\Gamma$, equipped with the
	(para)cyclic order induced from the corresponding edges of $\Gamma$.  
\end{prop}
\begin{proof}
	This is an immediate consequence of \cite[4.2.3.8]{lurie:htt} which allows us to compute the
	state sum colimit $X(\Gamma)$ by covering the state sum diagram with two subdiagrams whose
	corresponding colimits yield $X(\Gamma'')$ and $X(\Gamma')$, respectively.
\end{proof}

\begin{exa} \label{prop:ribpushouts}
	The following examples of pushout squares as given by Proposition \ref{prop:descent} will be
	used below:
	\begin{enumerate}
		\item 
		\begin{equation}\label{eq:push1}
		\begin{tikzpicture}[auto,every loop/.style={in=130,out=50,looseness=70}]

				\node (M21) {$X(\;$\tikz[baseline=(current bounding
				box.center),scale=0.7]{
			  \filldraw[black] (0,0) circle(1.5pt);
			  \draw (0:0) -- (30:.5);
			  \draw (0:0) -- (150:.5);
			  \draw (0:0) -- (270:.5);
			}$\;)$}; 

		\node (M22) [right= of M21] {$X(\;$\tikz[baseline=(current bounding box.center),scale=0.7]{
			  \filldraw[black] (0,0) circle(1.5pt);
			  \filldraw[black] (0,0.8) circle(1.5pt);
			  \draw (0,.4) circle(0.4);
			  \draw (0,0) -- (0,-.5); 
		}$\;)$}; 

		\node (M11) [above= of M21] {$X(\;$\tikz[baseline=(current bounding box.center),scale=0.7]{
			  \filldraw[black] (-.2,0) circle(1.5pt);
			  \filldraw[black] (0.2,0) circle(1.5pt);
			  \draw (-.2,-.5) -- (-.2,0) -- (-.2,.5);
			  \draw (.2,-.5) -- (.2,0) -- (.2,.5);
		}$\;)$};

		\node (M12) at (M22 |- M11) {$X(\;$\tikz[baseline=(current bounding box.center),scale=0.7]{
			  \filldraw[black] (0,0) circle(1.5pt);
			  \draw (0,.5) -- (0,0) -- (0,-.5); 
			}$\;)$};

		\draw [->] (M11) -- (M12);
		\draw [->] (M11) -- (M21);
		\draw [->] (M21) -- (M22);
		\draw [->] (M12) -- (M22);

		\end{tikzpicture}
	\end{equation}
\item \begin{equation}\label{eq:push2}
\begin{tikzpicture}[auto,every loop/.style={in=130,out=50,looseness=70}]

	\node (N11) {$X(\;$\tikz[baseline=(current bounding box.center),scale=0.7]{
		  \filldraw[black] (0,0) circle(1.5pt);
		  \draw (0,-.5) -- (0,0) -- (0,.5);
		}$\;)$};

	\node (N12) [right= of N11] {$X(\;$\tikz[baseline=(current bounding box.center),scale=0.7]{
			  \filldraw[black] (270:.4) circle(1.5pt);
			  \draw (0,0) circle(0.4);
			  \draw (270:.4) -- (270:.9);
			  \draw (90:.1) -- (270:.4);
	  }$\;)$}; 

	\node (N21) [below= of N11] {$X(\;$\tikz[baseline=(current bounding box.center),scale=0.7]{
		  \filldraw[black] (0,.5) circle(1.5pt);
		  \draw (0,0) -- (0,.5);
		}$\;)$}; 

	\node (N22) at (N12 |- N21)  {$X(\;$\tikz[baseline=(current bounding box.center),scale=0.7]{
		  \filldraw[black] (0,0) circle(1.5pt);
		  \filldraw[black] (0,0.4) circle(1.5pt);
		  \draw (0,.4) circle(0.4);
		  \draw (0,0) -- (0,-.5); 
			  \draw (0,0) -- (0,.4); 
		}$\;)$}; 

	\draw [->] (N11) -- (N12);
	\draw [->] (N11) -- (N21);
	\draw [->] (N21) -- (N22);
	\draw [->] (N12) -- (N22);

	\end{tikzpicture}
	\end{equation}
	\item Let $\Gamma$ be a ribbon (framed) graph which contains the graph
		
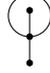
\[
		\tikz[yshift=4ex,scale=0.7]{
			  \filldraw[black] (270:.4) circle(1.5pt);
			  \filldraw[black] (270:.9) circle(1.5pt);
			  \filldraw[black] (90:.1) circle(1.5pt);
			  \draw (0,0) circle(0.4);
			  \draw (270:.4) -- (270:.9);
			  \draw (90:.1) -- (270:.4);
		}
	\]
	as a subgraph and so that the central vertex has valency $4$ in $\Gamma$. Let $\Gamma'$
	denote the ribbon (framed) graph obtained by removing the graph
		
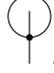
\[
		\tikz[yshift=4ex,scale=0.7]{
			  \filldraw[black] (270:.4) circle(1.5pt);
			  \draw (0,0) circle(0.4);
			  \draw (270:.4) -- (270:.9);
			  \draw (90:.1) -- (270:.4);
		}.
	\]
	Then we have a pushout square
	\begin{equation}\label{eq:push3}
	\begin{tikzpicture}[auto]
		\node (A11)  {$X(\;$\tikz[baseline=(current bounding box.center),scale=0.7]{
			\filldraw[black] (-.2,0) circle(1.5pt);
			\filldraw[black] (.2,0) circle(1.5pt);
			\draw (-.2,-.5) -- (-.2,0) -- (-.2,.5);
			\draw (.2,-.5) -- (.2,0) -- (.2,.5);
		}$\;)$};

		\node (A12) [right= of A11] {$X(\;$\tikz[baseline=(current bounding box.center),scale=0.7]{
			  \filldraw[black] (270:.4) circle(1.5pt);
			  \draw (0,0) circle(0.4);
			  \draw (270:.4) -- (270:.9);
			  \draw (90:.1) -- (270:.4);
		}$\;)$}; 

		\node (A21) [below= of A11] {$X(\Gamma')$};

		\node (A22) at (A12 |- A21) {$X(\Gamma)$.}; 

		\draw [->] (A11) -- (A12);
		\draw [->] (A11) -- (A21);
		\draw [->] (A21) -- (A22);
		\draw [->] (A12) -- (A22);
	\end{tikzpicture}
	\end{equation}
	\item \label{exa:local}
	The pushout square 
	\begin{equation}\label{eq:push4}
	\begin{tikzpicture}[auto]
		\node (A11)  {$X(\;$\tikz[baseline=(current bounding box.center),scale=0.7]{
			\filldraw[black] (-.2,0) circle(1.5pt);
			\filldraw[black] (.2,0) circle(1.5pt);
			\draw (-.2,-.5) -- (-.2,0) -- (-.2,.5);
			\draw (.2,-.5) -- (.2,0) -- (.2,.5);
		}$\;)$};

		\node (A12) [right= of A11] {$X(\;$\tikz[baseline=(current bounding box.center),scale=0.7]{
			  \filldraw[black] (270:.4) circle(1.5pt);
			  \draw (0,0) circle(0.4);
			  \draw (270:.4) -- (270:.9);
			  \draw (90:.1) -- (270:.4);
		}$\;)$}; 

		\node (A21) [below= of A11] {$X(\;$\tikz[baseline=(current bounding box.center),scale=0.7]{
		  \filldraw[black] (-.2,.4) circle(1.5pt);
		  \draw (-.2,0) -- (-.2,.4);
		  \filldraw[black] (.2,0) circle(1.5pt);
		  \draw (.2,0) -- (.2,.4);
		}$\;)$};

		\node (A22) at (A12 |- A21) {$X(\;$\tikz[baseline=(current bounding box.center),scale=0.7]{
			  \filldraw[black] (270:.4) circle(1.5pt);
			  \filldraw[black] (90:.1) circle(1.5pt);
			  \filldraw[black] (270:.9) circle(1.5pt);
			  \draw (0,0) circle(0.4);
			  \draw (270:.4) -- (270:.9);
			  \draw (90:.1) -- (270:.4);
		}$\;)$}; 

		\draw [->] (A11) -- (A12);
		\draw [->] (A11) -- (A21);
		\draw [->] (A21) -- (A22);
		\draw [->] (A12) -- (A22);
	\end{tikzpicture}
	\end{equation}
	is a special case of \eqref{eq:push3}.
\end{enumerate}
\end{exa}

We use Proposition \ref{prop:descent} to obtain localization sequences for topological Fukaya
categories.

\begin{prop}[Localization]\label{prop:localization} Let $\Gamma$ be a ribbon graph and $\Gamma' \subset \Gamma$ an open subgraph with
	complement $\Gamma''$. 
	\begin{enumerate}
		\item There is a canonical pushout square 
			\begin{equation}\label{eq:local}
			\xymatrix{
				\Ft(\Gamma') \ar[r]\ar[d] & \Ft(\Gamma) \ar[d]\\
				0 \ar[r] & \Ft(\underline{\Gamma''})
			}
			\end{equation}
			in $\Lmot$. 
		\item Assume that $\Gamma$ carries the structure of a framed graph. Then there is a
			lift of \eqref{eq:local} to a pushout diagram of $\ZZ$-graded Fukaya
			categories 
			\[
				\xymatrix{
					\Ff(\Gamma') \ar[r]\ar[d] & \Ff(\Gamma) \ar[d]\\
					0 \ar[r] & \Ff(\underline{\Gamma''})
				}
			\]
			in $\Lmo$. 
	\end{enumerate}
\end{prop}
\begin{proof} We treat the framed case, the ribbon case is analogous.
	By Proposition \ref{prop:descent}, we have a pushout square 
	\begin{equation}\label{eq:p1}
			\xymatrix{
				\Ff(\Xi) \ar[r]\ar[d] & \Ff(\Gamma'') \ar[d]\\
				\Ff(\Gamma') \ar[r] & \Ff(\Gamma).
			}
	\end{equation}
	Another application of Proposition \ref{prop:descent} yields a pushout square 
	\begin{equation}\label{eq:p2}
			\xymatrix{
				\Ff(\Xi) \ar[r]\ar[d] & \Ff(\Gamma'') \ar[d]\\
				\Ff(\Psi) \ar[r] & \Ff(\overline{\Gamma''})
			}
	\end{equation}
	where $\Psi$ denotes a disjoint union of copies of the corolla
	\[
		\tikz[yshift=2ex,scale=0.7]{
			\filldraw[black] (-.2,0) circle(1.5pt);
			\draw (-.2,-.5) -- (-.2,0);
		}
	\]
	again indexed by those half-edges in $\Gamma'$ which become internal in $\Gamma$.  
	We have $\Ff(\Psi) \simeq 0$ so that, using the universal property of the pushout
	\eqref{eq:p1}, we may combine \eqref{eq:p1} and \eqref{eq:p2} to obtain a canonical diagram 
	\[
			\xymatrix{
				\Ff(\Xi) \ar[r]\ar[d] & \Ff(\Gamma'') \ar[d]\\
				\Ff(\Gamma') \ar[r]\ar[d] & \ar[d] \Ff(\Gamma)\\
				0 \ar[r] & \Ff(\overline{\Gamma''}).
			}
	\]
	The top and exterior square are pushouts so that, by \cite[4.4.2.1]{lurie:htt}, the bottom
	square is a pushout square as well. To obtain the final statement, we note that, by the
	$2$-Segal property of $\Ff$, we have an equivalence $\Ff(\overline{\Gamma''}) \simeq
	\Ff(\underline{\Gamma''})$.
\end{proof}

\begin{rem} Note that the argument of Proposition \ref{prop:localization} generalizes to any
	co(para)cyclic $2$-Segal object $X$ with values in a {\em pointed} $\infty$-category
	satisfying $X^0 \simeq 0$.
\end{rem}

\begin{rem} Define $\Ff(\Gamma \; \text{on} \; \Gamma')$ to be  the full dg subcategory of
	$\Ff(\Gamma)$ spanned by the image of $\Ff(\Gamma')$. We call $\Ff(\Gamma \; \text{on} \;
	\Gamma')$ the {\em topological Fukaya category of $\Gamma$ with support in $\Gamma'$}. Then,
	in the terminology of Section \ref{sec:exact}, we have an exact sequence
	\[
		\xymatrix{
			\Ff(\Gamma \; \text{on}\; \Gamma') \ar[r]\ar[d] & \Ff(\Gamma) \ar[d]\\
			0 \ar[r] & \Ff(\underline{\Gamma''})
		}
	\]
	which should be regarded as an analog of Thomason-Trobaugh's proto-localization sequence
	for derived categories of perfect complexes in the context of algebraic K-theory of schemes.
\end{rem}

\subsection{A Mayer-Vietoris theorem}
\label{sec:motivic}

In this section, we use localization sequences for topological Fukaya categories to prove a
Mayer-Vietoris theorem. The argument is inspired by similar techniques in the context of algebraic
K-theory for schemes \cite{thomason-trobaugh}. We focus on the $\ZZ$-graded version but all results
have obvious $\Zt$-graded analogs which admit similar proofs. 

\begin{defi} Let $H:\Lmo \to \C$ be a functor with values in a stable $\infty$-category $\C$.
	\begin{enumerate}
		\item We say $H$ is {\em localizing} if it preserves finite sums and exact sequences.
		\item The functor $H$ is called {\em $\AA^1$-homotopy invariant} if, for every dg category $A$, the morphism 
			\[
				A \lra A[t]
			\]
			is an $H$-equivalence (i.e, is mapped to an equivalence under $H$). Here,
			the morphism $A \to A[t]$ is defined as the tensor product of $A$ with the
			morphism $k \to k[t]$ of $k$-algebras, interpreted as a dg functor of dg
			categories with one object. 
	\end{enumerate}
\end{defi}

Let $\C$ be an $\infty$-category with finite colimits and let $H: \Lmo \to \C$ be a functor, we have a
canonical morphism
\begin{equation}\label{eq:canmor}
	(HF)(\Gamma) \lra H(F(\Gamma)) 
\end{equation}
in $\C$ which is obtained by applying $H$ to the colimit cone over the diagram $F \circ \delta$.

\begin{defi} 
	We say a framed graph $\Gamma$ {\em satisfies Mayer-Vietoris with respect to $H$} if the
	morphism \eqref{eq:canmor} is an equivalence. In other words, a framed graph $\Gamma$
	satisfies Mayer-Vietoris with respect to $H$, if $H$ commutes with the state sum colimit
	parametrized by the incidence category of $\Gamma$. 
\end{defi}

\begin{thm}[Mayer-Vietoris]\label{thm:motivicMV}
  Let $\C$ be a stable $\infty$-category and let $H: \Lmo \to \C$ be a localizing $\AA^1$-homotopy
  invariant functor. Then every framed graph satisfies Mayer-Vietoris with respect to $H$.
\end{thm}
\begin{proof}
	The proof will use the results from Section \ref{sec:lemmas} below. 
	Note that, by definition, every framed corolla satisfies Mayer-Vietoris. By Lemma 
	\ref{lem:hpush}, with respect to \eqref{eq:push1}, and Lemma \ref{lem:induct}(2), we deduce
	that the graph 
	
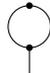
\[
		\tikz[scale=0.7,baseline=(current bounding box.center)]{
			  \filldraw[black] (0,0) circle(1.5pt);
			  \filldraw[black] (0,.8) circle(1.5pt);
			  \draw (0,.4) circle(0.4);
			  \draw (0,0) -- (0,-.5); 
		}
	\]
	provided with any framing, satisfies Mayer-Vietoris. By Proposition \ref{thm:funcontract}, we may contract an internal
	edge to obtain that 
	
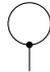
\[
		\tikz[scale=0.7,baseline=(current bounding box.center)]{
			  \filldraw[black] (0,0) circle(1.5pt);
			  \draw (0,.4) circle(0.4);
			  \draw (0,0) -- (0,-.5); 
		}
	\]
	satisfies Mayer-Vietoris.
	Similarly, by Lemma \ref{lem:hpush}, with respect to \eqref{eq:push2}, and Lemma \ref{lem:induct}(2), we obtain
	that the graph
	
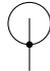
\[
		\tikz[scale=0.7,baseline=(current bounding box.center)]{
			  \filldraw[black] (270:.4) circle(1.5pt);
			  \draw (0,0) circle(0.4);
			  \draw (270:.4) -- (270:.9);
			  \draw (90:.1) -- (270:.4);
	  	}
	\]
	with any framing, satisfies Mayer-Vietoris. Here, we again use Proposition \ref{thm:funcontract} to contract
	the internal edge of the graph in the bottom right corner of \eqref{eq:push2}. 

	Let $\Gamma$ be any framed graph. Since $H$ commutes with finite sums, we may assume that
	$\Gamma$ is connected. By Proposition \ref{thm:funcontract}, $\Gamma$ satisfies
	Mayer-Vietoris if and only if the graph $\Gamma'$ with one vertex obtained by collapsing a maximal forest in
	$\Gamma$ satisfies Mayer-Vietoris. Therefore, we may assume that $\Gamma$ has one vertex $v$. 
	We now proceed inductively on the number of loops in $\Gamma$. If $\Gamma$ does not have any loops
	then $\Gamma$ is a corolla and satisfies Mayer-Vietoris. Assume $\Gamma$ has a loop $l$. We
	isolate the loop $l$ by blowing up two edges so that we obtain a graph $\widetilde{\Gamma}$ with three vertices
	which contains the subgraph  
		
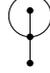
\[
		\tikz[yshift=4ex,scale=0.7]{
			  \filldraw[black] (270:.4) circle(1.5pt);
			  \filldraw[black] (270:.9) circle(1.5pt);
			  \filldraw[black] (90:.1) circle(1.5pt);
			  \draw (0,0) circle(0.4);
			  \draw (270:.4) -- (270:.9);
			  \draw (90:.1) -- (270:.4);
		}
	\]
	with loop $l$, satisfies the conditions of Example \ref{prop:ribpushouts}(3), and $\Gamma$
	is obtained from $\widetilde{\Gamma}$ by contracting the two internal edges. 
	We now apply the induction hypothesis, Lemma 
	\ref{lem:hpush} with respect to \eqref{eq:push3}, and Lemma \ref{lem:induct}(1) to deduce
	that $\widetilde{\Gamma}$ and hence, by Proposition \ref{thm:funcontract}, $\Gamma$
	satisfies Mayer-Vietoris, concluding the argument.
\end{proof}

\begin{rem} Let $k$ be a field of characteristic $0$. An example of a localizing functor for which 
	Mayer-Vietoris fails (since $\AA^1$-homotopy invariance does not hold) is given by Hochschild homology. Applying $HH_* \circ j_! F$
	to the square \eqref{eq:push1}, we obtain 
	  \[
	    \begin{tikzcd}
	      HH_*(k \amalg k) \arrow{r} \arrow{d}{f} & HH_*(k) \arrow{d}{f'}\\
	      HH_*(A^1) \arrow{r} & HH_*(k[t]) 
	    \end{tikzcd}
	  \]
	where the left vertical morphism becomes an equivalence by Proposition \ref{prop:exact}.
	However, the right vertical morphism is {\em not} an equivalence: By the
	Hochschild-Kostant-Rosenberg theorem, we have $HH_1(k[t]) \cong k[t]$ while $HH_1(k) \cong 0$.
	In this situation, replacing Hochschild homology by periodic cyclic homology leads to an
	$\AA^1$-homotopy invariant functor which therefore satisfies Mayer-Vietoris.
\end{rem}

\subsection{Lemmas}
\label{sec:lemmas}

We collect technical results for the proof of the Mayer-Vietoris theorem.

\begin{prop}\label{prop:exact} Let $\C$ be a stable $\infty$-category and let $H:\Lmo \to \C$ be
	a localizing functor. Then the coparacyclic object 
  \[
	  HF: \N(\Lambda_{\infty}) \lra \C
  \]
  is $1$-Segal.
\end{prop}
\begin{proof}
  We have to show that, for every $n \ge 1$, the natural map
  \[
    H(\Ff^{\{0,1\}}) \amalg H(\Ff^{\{1,2\}}) \amalg \dots \amalg H(\Ff^{\{n-1,n\}}) \to
    H(\Ff^{\{0,1,\dots,n\}})
  \]
  is an equivalence in $\C$. We have a diagram in $\Lmo$
  \[
    \xymatrix{
      \Ff^{\{0,1,\dots,n-1\}} \ar[r]\ar[d] & \Ff^{\{0,1,\dots,n-1\}} \amalg \Ff^{\{n-1,n\}}
      \ar[r]\ar[d]^f & \Ff^{\{n-1,n\}}\ar[d]\\
      \Ff^{\{0,1,\dots,n-1\}} \ar[r] & \Ff^{\{0,1,\dots,n\}} \ar[r] & \Ff^{\{n-1,n\}}}
  \]
  with exact rows. After applying $H$, the rows stay exact, and $H(f)$ must be an equivalence, since
  its cofiber is $0$. We now proceed by induction on $n$.
\end{proof}

\begin{lem}\label{lem:devissage}
	Let $k[\delta]$ denote the differential $\ZZ$-graded $k$-algebra generated by $\delta$ in some degree with relation $\delta^2 =
	0$ and zero differential. Let $A$ be a dg category and consider a pushout diagram
	\begin{equation}\label{eq:pushdelta}
	    \begin{tikzcd}
	      k \arrow{r}\arrow{d} & k[\delta]\arrow{d}\\
	      A \arrow{r}{i} & A^\delta
	    \end{tikzcd}
	  \end{equation}
	so that $A^\delta$ is obtained from $A$ by adjoining the endomorphism $\delta$ to a fixed object in $A$.
	Then, for any localizing $\AA^1$-homotopy invariant $H: \Lmo \to \C$, the morphism $H(i)$ is an equivalence.
\end{lem}
\begin{proof}
	We may extend \eqref{eq:pushdelta} to the diagram 
	\begin{equation}\label{eq:splitting}
	    \begin{tikzcd}
	      k \arrow{r}\arrow{d} & k[\delta]\arrow{d} \arrow{r}{\delta \mapsto 0}& k\arrow{d}\\
	      A \arrow{r}{i} & A^\delta \arrow{r}{p} & A
	    \end{tikzcd}
	  \end{equation}
	in which both squares are pushout squares so that $pi \simeq \id_A$. We will construct a
	commutative diagram
	\begin{equation}\label{eq:a1homotopy}
		\begin{tikzcd} 
			& A^\delta\\
			A^\delta \arrow{r}{h}\arrow{ur}{\on{id}} \arrow[swap]{dr}{i \circ p} &
			A^\delta[t]\arrow{u}[swap]{\ev_1} \arrow{d}{\ev_0}\\
			& A^\delta
		\end{tikzcd}
	\end{equation}
	in $\Lmo$ where $\ev_0$ and $\ev_1$ are $H$-equivalences. This will imply that $i \circ p$
	is an $H$-equivalence so that the diagram exhibits $p: A^\delta \to A$ as an 
	$\AA^1$-homotopy inverse of $i: A \to A^\delta$. In particular, we have that $i$ is an
	$H$-equivalence.

	To construct \eqref{eq:a1homotopy}, consider the morphism $\lambda: k[\delta] \to
	k[\delta,t]$ determined by $\lambda(\delta) = \delta t$. Further, let $g: k[\delta,t] \to
	A^\delta[t]$ be the morphism obtained by tensoring $k[\delta] \to A^\delta$ with $k[t]$.
	Further, let $f: A \to A^\delta[t]$ be the composite of the bottom horizontal morphisms in
	the diagram
	\[
	    \begin{tikzcd}
		    k \arrow{r}\arrow{d} & k[t]\arrow{d}\arrow{r} & k[\delta,t]\arrow{d}\\
		    A \arrow{r} & A[t] \arrow{r} & A^\delta[t]
	    \end{tikzcd}
	\]
	where all morphisms are the apparent ones.
	By construction, the above morphisms fit into a cone diagram
	\[
	    \begin{tikzcd}
		    k \arrow{r}\arrow{d} & k[\delta]\arrow{d} \arrow[bend left]{ddr}{g \lambda} & {}\\
		    A \arrow{r}{i} \arrow[bend right]{drr}[swap]{f} & A^\delta \arrow[dashed]{dr}{h} & {} \\
		    {} &{} & A^\delta[t]
	    \end{tikzcd}
	\]
	which determines the dashed morphism $h$ (up to contractible choice). It is now immediate to
	verify that the morphism $h$ fits into a diagram of the form \eqref{eq:a1homotopy} where
	$\ev_0$ and $\ev_1$ are the morphisms obtained by applying $A^\delta \otimes -$ to the
	morphisms $k[t] \to k, t \mapsto 0$, and $k[t] \to k, t \mapsto 1$, respectively.
\end{proof}

\begin{lem}\label{lem:induct} Let $\C$ be a stable $\infty$-category, and let $H: \Lmo \to \C$ be
	a localizing functor. Suppose 
	\begin{equation}\label{eq:ribp}
			\begin{tikzcd}
				\Lambda_{\infty}(\Gamma_1) \arrow{r}\arrow{d} & \Lambda_{\infty}(\Gamma_2) \arrow{d}\\
				\Lambda_{\infty}(\Gamma_3) \arrow{r} & \Lambda_{\infty}(\Gamma_4)
			\end{tikzcd}
	\end{equation}
	is a pushout diagram in $\P(\Lambda_{\infty})$ which stays a pushout diagram after application of $H
	\circ j_! \Ff$.
	\begin{enumerate}
		\item Assume that $\Gamma_1$, $\Gamma_2$, and $\Gamma_3$ satisfy Mayer-Vietoris.
			Then $\Gamma_4$ satisfies Mayer-Vietoris.
		\item Suppose $F({\Gamma_3}) \simeq 0$ and assume that any two among the ribbon
			graphs $\Gamma_1$, $\Gamma_2$, $\Gamma_4$ satisfy Mayer-Vietoris. 
			Then all ribbon graphs in \eqref{eq:ribp} satisfy Mayer-Vietoris.
	\end{enumerate}
\end{lem}
\begin{proof}
	Since $H \circ j_! \Ff$ is a left extension of $H\Ff: \N(\Lambda_{\infty}) \to \C$ along $j:
	\Lambda_{\infty} \to \P(\Lambda_{\infty})$, we have a canonical morphism $\xi: j_!(H\Ff) \to H \circ j_! \Ff$ in
	$\Fun(\P(\Lambda_{\infty}),\C)$ which, evaluated at
	an object of the form $\Lambda_{\infty}(\Gamma)$ yields the canonical morphism
	\[
		(H\Ff)(\Gamma) \to H(\Ff(\Gamma))
	\]
	in $\C$ of \eqref{eq:canmor}. We give the argument for (1). 
	Applying $\xi$ to the square \eqref{eq:ribp}, we obtain a morphism of pushout squares in $\C$ 
	which is an equivalence on all vertices except for the bottom right vertex
	\[
		(H\Ff)(\Gamma_4) \to H(\Ff(\Gamma_4)).
	\]
	But pushouts of equivalences are equivalences so that this morphism must be an equivalence
	as well. The statement of (2) follows similarly.
\end{proof}

\begin{lem}\label{lem:hpush} Let $H: \Lmo \to \C$ be a localizing $\AA^1$-homotopy invariant. The
	pushout diagrams \eqref{eq:push1}, \eqref{eq:push2}, and \eqref{eq:push3} (cf. Proposition
	\ref{prop:ribpushouts}), with $X=j: \N(\Lambda_{\infty}) \to \P(\Lambda_{\infty})$, stay
	pushout diagrams after application of $H \circ j_!\Ff$.
\end{lem}
\begin{proof}
	 (1) Applying $j_!\Ff$ to \eqref{eq:push1}, with $X =j$, we obtain, by Example \ref{exa:kt}, the pushout diagram
	  \[
	    \begin{tikzcd}
	      \Ff^1 \amalg \Ff^1 \arrow{r} \arrow{d}{f} & \Ff^1 \arrow{d}{f'}\\
	      \Ff^2 \arrow{r} & \Ff^1[t] 
	    \end{tikzcd}
	  \]
  	where $f$ is a $1$-Segal map so that, by Proposition \ref{prop:exact}, $H(f)$ is an equivalence. The
	morphism $H(f')$ is an equivalence, by the argument of \cite[Lemma
	4.1]{tabuada:fundamental}, since $H$ is $\AA^1$-homotopy invariant. Therefore, the
	square remains a pushout square upon applying $H$. The statement for \eqref{eq:push2}
	follows from an argument similar to (2) below, but easier.
	
(2) Applying $j_!\Ff$ to \eqref{eq:push3}, with $X=j$, we obtain the pushout square 
	\[
	\begin{tikzpicture}[auto]
		\node (A11)  {$\Ff^1 \amalg \Ff^1$};

		\node (A12) [right= of A11] {$\Ff(\;$\tikz[scale=0.7,baseline=(current bounding
		box.center)]{
			  \filldraw[black] (270:.4) circle(1.5pt);
			  \draw (0,0) circle(0.4);
			  \draw (270:.4) -- (270:.9);
			  \draw (90:.1) -- (270:.4);
		}$\;)$}; 

		\node (A21) [below= of A11] {$\Ff(\Gamma')$};

		\node (A22) at (A12 |- A21) {$\Ff(\Gamma)$}; 

		\draw [->] (A11) -- (A12);
		\draw [->] (A11) -- (A21);
		\draw [->] (A21) -- (A22);
		\draw [->] (A12) -- (A22);
	\end{tikzpicture}
	\]
	which is the top rectangle of the diagram 
	\[
	\begin{tikzpicture}[auto]
		\node (A11)  {$\Ff^1 \amalg \Ff^1$};

		\node (A11p) [right= of A11]  {$(\Ff^1)^{\delta_1} \amalg (\Ff^1)^{\delta_2}$};

		\node (A12) [right= of A11p] {$\Ff(\;$\tikz[baseline=(current bounding box.center),
		scale=0.7]{
			  \filldraw[black] (270:.4) circle(1.5pt);
			  \draw (0,0) circle(0.4);
			  \draw (270:.4) -- (270:.9);
			  \draw (90:.1) -- (270:.4);
		}$\;)$};

		\node (A21) [below= of A11] {$\Ff(\Gamma')$};

		\node (A21p) at (A11p |- A21) {$((\Ff(\Gamma'))^{\delta_1})^{\delta_2}$};

		\node (A22) at (A12 |- A21) {$\Ff(\Gamma)$}; 

		\node (A31) [below= of A21] {$0$};

		\node (A31p) at (A11p |- A31) {$0$};

		\node (A32) at (A12 |- A31) {$\Ff(\;$\tikz[baseline=(current bounding box.center),scale=0.7]{
			  \filldraw[black] (270:.4) circle(1.5pt);
			  \filldraw[black] (90:.1) circle(1.5pt);
			  \filldraw[black] (270:.9) circle(1.5pt);
			  \draw (0,0) circle(0.4);
			  \draw (270:.4) -- (270:.9);
			  \draw (90:.1) -- (270:.4);
		}$\;)$}; 

		\draw [->] (A11) -- (A11p) node[midway,above] {$f$};
		\draw [right hook->] (A11p) -- (A12) node[midway,above] {$g$};
		\draw [->] (A11) -- (A21);
		\draw [->] (A11p) -- (A21p);
		\draw [->] (A21) -- (A21p) node[midway,above] {$f'$};
		\draw [right hook->] (A21p) -- (A22) node[midway,above] {$g'$};
		\draw [->] (A12) -- (A22);
		\draw [->] (A31) -- (A31p);
		\draw [->] (A31p) -- (A32);
		\draw [->] (A21) -- (A31);
		\draw [->] (A21p) -- (A31p);
		\draw [->] (A22) -- (A32);
	\end{tikzpicture}
	\]
	in which all squares are pushout squares, the exterior square is given \eqref{eq:push4} with
	$X = \Ff$. We use the notation from Lemma \ref{lem:devissage}
	with $|\delta_1| = 1-2r$, $|\delta_2| = 1+2r$ where $r$ is the winding number of the loop in
	the top right framed graph.
  	The morphisms $H(f)$ and $H(f')$ are equivalences by Lemma \ref{lem:devissage}. 
	The morphism $g$ is quasi-fully faithful by explicit verification (cf. Example
	\ref{ex:kroneckerZ}), and hence $g'$ is quasi-fully
	faithful by Lemma \ref{lem:qffpushout}. Therefore, the bottom right square and the right rectangle are exact
  	sequences in $\Lmo$ and thus stay exact after applying $H$. But this implies that the top
	right square stays a pushout square after applying $H$. Hence the top rectangle stays a
	pushout diagram after applying $H$ as claimed.
\end{proof}

\section{Delooping of paracyclic $1$-Segal objects}

The methods used in this section are inspired by similar techniques used in Chapter 6 of
\cite{lurie:htt}.
Let $\C$ be a stable $\infty$-category and let $H:\Lmo \to \C$ be a localizing $\AA^1$-homotopy invariant functor.
Given a framed graph $\Gamma$, Theorem \ref{thm:motivicMV} reduces the calculation of
$H(F^{\Gamma})$ to a state sum of the coparacyclic object
\[
	HF: \N(\pLambda) \lra \C
\]
which, by Proposition \ref{prop:exact}, is $1$-Segal. We establish a general result in order to
explicitly evaluate this state sum.

\begin{thm}\label{thm:segalstatesum} Let $\C$ be a stable $\infty$-category with colimits and let $X: \N(\pLambda) \to \C$ be a
	coparacyclic $1$-Segal object with $X^0 \simeq 0$. Let $\Gamma$ be a framed graph modelling a 
	stable framed marked surface $(S,M)$. Then we have an equivalence in $\C$
	\[
		X({\Gamma}) \simeq \Omega(X^1)(S/M) 
	\]
	where the right-hand side is defined using \eqref{eq:E}, and
	$\Omega =[-1]$ denotes the loop functor on $\C$.
\end{thm}

The results of Section \ref{sec:universal} reduce the proof of Theorem \ref{thm:segalstatesum} 
to the following statement:

\begin{prop}[Delooping]\label{prop:paraloop} Let $\C$ be a stable $\infty$-category. Let $\C$ be a stable $\infty$-category with
	colimits. Let $X: \N(\pLambda) \to \C$ be a coparacyclic $1$-Segal object with $X^0 \simeq
	0$. Then there is an equivalence 
	\[
		X \simeq L_{\Omega(X^1)}
	\]
	where the right-hand side denotes the coparacyclic object of \eqref{eq:universal}.
\end{prop}

\begin{rem} Let $M$ be a monoid object in an abelian category $\A$. Then it is easy to see that 
	the monoid structure must be given by the unit $0 \lra M$ and multiplication $M \otimes M
	\to M$ given by the sum. In particular, $M$ is automatically a group object with inverse
	given by $-\id: M \to M$. Proposition \ref{prop:paraloop} is an $\infty$-categorical analog of
	this statement where the abelian category $\A$ gets replaced by the stable $\infty$-category
	$\C^{\op}$. In this context, the statement becomes even nicer since we can describe the group object 
	in terms of the delooping $\Omega X^1$. We also note that the statement of Proposition
	\ref{prop:paraloop} becomes false if we replace the paracyclic category by the cyclic
	category -- there may be various inequivalent cocyclic structures on a cosimplicial $1$-Segal
	object. 
\end{rem}

We establish some preparatory results for the proof of Proposition \ref{prop:paraloop}.

\begin{prop}\label{prop:lim} Let $\C$ be a stable $\infty$-category and let $X: \N(\Delta) \to \C$ be a cosimplicial
	$1$-Segal object. Then $\lim_{\N(\Delta)} X \simeq \Omega X_1$.
\end{prop}
\begin{proof} Let $\I$ be the full subcategory of $\Delta$ spanned by the objects $[0]$ and $[1]$.
	Let $\I_{+}$ denote the category obtained from $\I$ by adjoining an initial object. We have
	a diagram of right Kan extension functors
	\begin{equation}\label{eq:right-kan}
			\xymatrix{
				\Fun(\N(\I), \C) \ar[r]\ar[d] & \Fun(\N(\I_+), \C)\ar[d]\\
			\Fun(\N(\Delta), \C) \ar[r] & \Fun(\N(\Delta_+), \C) }
	\end{equation}
	which commutes up to equivalence. We show that $X$ is a right Kan extension of
	$X_{|\N(\I)}$. The statement of the proposition then follows immediately from
	\eqref{eq:right-kan}. 
	
	Since $X^0 \simeq 0$, the cosimplicial object $X$ is a right Kan extension of $X_{|\N(\I)}$ if and only if, for
	every $n \ge 2$, the maps $X_n \to X_1$ induced by the $n$ surjective maps $p_i: [n] \mapsto
	[1]$ exhibit $X^n$ as a product
	\begin{equation}\label{eq:kan}
			X^n \simeq X^1 \times X^1 \times \dots \times X^1.
	\end{equation}
	with $n$ factors. This can be verified in the homotopy category $\h
	\C$ which is additive. Since $X$ is $1$-Segal, the maps $[1] \to [n]$ given by the inclusions $q_j:
	\{j,j+1\} \subset [n]$ exhibit $X^n$ as a coproduct of $n$ copies of $X^1$. The morphisms $X^1
	\to X^1$ induced by $p_i \circ q_j$ are homotopic to the identity on $X_1$ if $i = j$ and
	zero if $i \ne j$. This implies the equivalence \eqref{eq:kan}.
\end{proof}

Let $i: \Delta \to \pLambda$ denote the natural inclusion functor.

\begin{lem}\label{lem:cofinal} The opposite inclusion functor
	\[
		\N(i)^{\op}: \N(\Delta)^{\op} \lra \N(\pLambda)^{\op}
	\]
	is cofinal. 
\end{lem}
\begin{proof} By \cite[4.1.3.1]{lurie:htt}, it suffices to show that, for every object $\cm$ of
	$\pLambda$, the nerve of the category $\Delta/\cm$ has a contractible geometric realization.
	The nerve of the category $\Delta/\cm$ can be identified with the simplicial set 
	\[
		\Hom_{\pLambda}(-,\cm)_{|\Delta^{\op}}
	\]
	whose geometric realization is homeomorphic to $|\Delta^m| \times \mathbb R$ (cf.
	\cite{fiedorowicz-loday}).
\end{proof}

\begin{cor} \label{prop:limp} Let $\C$ be a stable $\infty$-category and let $X: \N(\pLambda) \to \C$
	be a coparacyclic $1$-Segal object. Then $\lim_{\N(\pLambda)} X \simeq \Omega X_1$.
\end{cor}
\begin{proof} Immediate from Proposition \ref{prop:lim} and Lemma \ref{lem:cofinal}.
\end{proof}

We denote by $\apLambda$ the {\em augmented paracyclic category} obtained from $\pLambda$ by adjoining an initial
object $\emptyset$. Consider the full subcategory $j: \N(\J) \subset \N(\apLambda)$ spanned by the objects $\emptyset$ 
and $\langle 0 \rangle$ and the map $p: \N(\J) \to \Delta^1$ which maps the unique edge $\emptyset \to \langle 0
\rangle$ to $\{0,1\}$. 


\begin{prop}\label{prop:nerve} Let $\C$ be a stable $\infty$-category and let $X: \N(\pLambda) \to \C$ be a
	coparacyclic $1$-Segal object with $X^0 \simeq 0$. Let $X^+$ denote the limit cone of $X$,
	i.e., the right Kan extension of $X$ along $\N(\pLambda) \to \N(\apLambda)$. Then there is an
	equivalence 
	\[
		X \simeq j_!(p^*(\Omega(X^1) \lra 0)).
	\]
	of coparacyclic objects. 
\end{prop}
\begin{proof}
	By Proposition \ref{prop:limp}, we have 
	\[
		X^+(\emptyset) \simeq \Omega(X^1)
	\]
	and by assumption $X^+(\langle 0 \rangle) \simeq 0$ so that we may exhibit $X^+$ is a left extension of
	$p^*(\Omega(X^1) \lra 0)$ along $\N(\J) \to \N(\pLambda)$. We use the pointwise formula 
	to show that it is a left Kan extension. For the object $\langle 1 \rangle$, this amounts
	to the statement that the square
	\[
		\xymatrix{
			\Omega(X^1) \ar[r]\ar[d] & 0 \ar[d] \\
		0 \ar[r] & X^1  }
	\]
	is a pushout square which holds since $\C$ is stable. For the object $\cn$, the pointwise
	formula reduces to the requirement that the $n$ coface maps $X^1 \to X^n$ exhibit $X^n$ as an
	$n$-fold coproduct of copies of $X^1$. This holds since $X$ is assumed $1$-Segal.
\end{proof}

\begin{proof}[Proof of Proposition \ref{prop:paraloop}]
	Let $\pi: \pLambda \to \Lambda$ be the covering functor. Consider $L^{\bullet}$ from
	\eqref{eq:uniloop} as a coparacyclic $1$-Segal pointed space via pullback along $\pi$ (cf.
	Remark \ref{rem:universal}). Given an object $E \in \C$, the functor $E(-)$ from
	\eqref{eq:E} commutes with finite colimits so that the coparacyclic object $E(L^{\bullet})$
	in $\C$ is $1$-Segal. We have $E(L^1) \simeq E(S^1) \simeq E[1]$ so that, by Proposition
	\ref{prop:nerve}, we have 
	\[
		L_E \simeq E(L^{\bullet}) \simeq j_!(p^*(E \lra 0)).
	\]
	Combining this with another application of Proposition \ref{prop:nerve} to the given coparacyclic $1$-Segal
	object $X$, yields the desired equivalence
	\[
		X \simeq j_!(p^*(\Omega(X^1) \lra 0)) \simeq L_{\Omega(X^1)}.
	\]
\end{proof}

\section{Main result and examples}

Combining Theorem \ref{thm:motivicMV} and Theorem \ref{thm:segalstatesum}, we obtain the following main result of this work.

\begin{thm}\label{thm:main} Let $\C$ be a stable $\infty$-category with colimits and let
\[
	H: \Lmo \lra \C
\]
be a localizing $\AA^1$-homotopy invariant. Let $\Gamma$ be a framed graph which models a stable framed
marked surface $(S,M)$ and let $\Ff(S,M)$ denote its $\ZZ$-graded topological Fukaya category. Then
there is an equivalence
\[
	H(\Ff(S,M)) \simeq \Sigma^{\infty}(S/M) \otimes H(k)[-1].
\]
\end{thm}

As a specific application, we obtain:

\begin{cor} Let $k$ be a field of characteristic $0$. We have the formulas
	\begin{align*}
		\HP_0( \Ff(S,M)) & \cong  H_1(S,M;k)\\
		\HP_1( \Ff(S,M)) & \cong  H_2(S,M;k)
	\end{align*}
	for periodic cyclic homology over $k$.
\end{cor}

\begin{exa} We compute some explicit examples. 
	\begin{enumerate} 
		\item Let $(S,M) = (S^2, \{0,\infty\})$ be a $2$-sphere with two marked
			points. Then, equipping the punctured sphere with the standard framing with winding number
			$0$, we have $\Ff(S,M) \simeq \Perf(\AA^1 \setminus \{0\})$. In agreement with
			Hochschild-Kostant-Rosenberg, we obtain
			\begin{align*}
				\HP_0( \Ff(S,M)) & \cong  H_1(S,M;k) \cong k\\
				\HP_1( \Ff(S,M)) & \cong  H_2(S,M;k) \cong k.
			\end{align*}
		\item Let $(S,M) = (T, \{0\})$ be a once marked torus equipped with standard
			framing. We have a Morita equivalence
			\[
				\Ff(S,M) \simeq \D^b(\coh(C))
			\]
			where $C$ denotes a nodal plane cubic. We obtain
			\begin{align*}
				\HP_0( \Ff(S,M)) & \cong  H_1(S,M;k) \cong k^2\\
				\HP_1( \Ff(S,M)) & \cong  H_2(S,M;k) \cong k.
			\end{align*}
	\end{enumerate}
\end{exa}

\bibliographystyle{amsalpha} 
\bibliography{refs} 

\end{document}